\documentclass[a4paper,11pt]{amsart}
\usepackage{amssymb}
\usepackage{amsmath}
\usepackage{latexsym}
\usepackage{eepic}
\usepackage{epsfig}
\usepackage{graphicx}
\usepackage{amscd}
\usepackage{color}
\usepackage[bookmarks=false]{hyperref}
\usepackage{tikz}
\usepackage{fancyhdr}

\usepackage{eufrak}
\usepackage{mathrsfs}
\usepackage[english]{babel}
\usepackage[all,cmtip]{xy}
\DeclareMathAlphabet{\mathpzc}{OT1}{pzc}{m}{it}
\newtheorem{theorem}{Theorem}[section]
\newtheorem*{theorem*}{Theorem}
\newtheorem{theorem-definition}[theorem]{Theorem-Definition}
\newtheorem{lemma-definition}[theorem]{Lemma-Definition}
\newtheorem{definition-prop}[theorem]{Proposition-Definition}
\newtheorem{corollary}[theorem]{Corollary}
\newtheorem{prop}[theorem]{Proposition}
\newtheorem*{prop*}{Proposition}
\newtheorem{lemma}[theorem]{Lemma}
\newtheorem{cor}[theorem]{Corollary}
\newtheorem{definition}[theorem]{Definition}
\newtheorem*{definition*}{Definition}

\theoremstyle{definition}
\newtheorem{example}[theorem]{Example}

\newtheorem{remark}[theorem]{Remark}

\newcommand{\Q}{\ensuremath{\mathbb{Q}}}

\newcommand{\cX}{\ensuremath{\mathscr{C}}}

\newcommand{\cZ}{\ensuremath{\mathscr{Z}}}
\newcommand{\cM}{\ensuremath{\mathscr{M}}}

\newcommand{\cY}{\ensuremath{\mathscr{Y}}}

\newcommand{\cL}{\ensuremath{\mathscr{L}}}
\newcommand{\cC}{\ensuremath{\mathscr{C}}}
\newcommand\calo{\mathcal O}

\newcommand{\mcr}[1]{\mathscr{#1}}

\renewcommand{\cM}{\ensuremath{\mathscr{M}}}

\renewcommand{\cZ}{\ensuremath{\mathscr{Z}}}
\renewcommand{\cY}{\ensuremath{\mathscr{Y}}}

\newcommand{\Spec}{\ensuremath{\mathrm{Spec}\,}}

\newcommand{\Lie}{\mathrm{Lie}}

\newcommand{\Jac}{\mathrm{Jac}}

\newcommand{\tame}{\mathrm{tame}}
\newcommand{\wild}{\mathrm{wild}}

\newcommand{\coker}{\mathrm{coker}}

\newcommand{\Art}{\mathrm{Art}}

\numberwithin{equation}{section}

\title[Base change conductors]{Base change conductors through intersection theory and quotient singularities}
\author{Dennis Eriksson}
\author{Lars Halvard Halle}
\author{Johannes Nicaise}

\address{Dennis Eriksson \\ Department of Mathematics \\ Chalmers University of Technology and University of Gothenburg, Sweden}
\email{dener@chalmers.se}

\address{Lars Halvard Halle \\ Department of Mathematics\\
University of Bologna, Italy}
\email{larshalvard.halle@unibo.it}
\address{Johannes Nicaise \\ Department of Mathematics\\
KU Leuven, Belgium
}
\email{johannes.nicaise@kuleuven.be}
\begin{document}
\begin{abstract}

We perform a systematic study of the base change conductor for Jacobians. Through the lens of intersection theory and Deligne's Riemann–Roch theorem, we present novel computational approaches for both the tame and wild parts of the base change conductor. Our key results include a general formula of the tame part, as well as a computation of the wild part in terms of Galois quotients of semistable models of the curves. We treat in detail the case of potential good reduction when the quotient only has weak wild quotient singularities, relying on recent advances by Obus and Wewers.
\end{abstract}
\maketitle
\tableofcontents

\section{Introduction}
\subsection{}

An algebraic family of varieties over a punctured curve can be completed into a family over the whole curve. Such models of the family are not unique, and some are better behaved than others, for instance smooth and proper models. In the setting of an abelian variety $A$ defined over the fraction field $K$ of a Dedekind domain $R$, one can ask for models over $R$ that also extend the group structure. By Grothendieck's semi-abelian reduction theorem (cf. \cite[IX.3.6]{SGA71}) it is in fact possible to find a model of $A$ which forms a semi-abelian scheme over the base, if one allows a finite ramified extension. This turns out to be the best one can hope for in general.

In this article, we study an invariant called the \emph{base change conductor}, which, for a given abelian $K$-variety $A$, is a rational number $c(A)$ that measures the failure of $A$ to have semi-abelian reduction over $R$. Roughly, the value $c(A)$ is obtained by comparing the Lie algebra of the N\'eron model (which is not compatible with base change in general) before and after semiabelian reduction.

The base change conductor was first introduced by Chai and Yu in the context of tori \cite{chai-yu}, and more generally for (semi-)abelian varieties by Chai in \cite{chai}, see \ref{subsec:basechangedef} for the definition. It shows up, directly or indirectly, in many arithmetic problems of abelian varieties. For instance, if $ A $ is an abelian variety over a number field, the base change conductor appears as a correction factor when comparing the Faltings height (which is defined in terms of the Lie algebra of the N\'eron model, cf. \cite{Faltings}) computed before and after making an extension over which $A$ acquires semi-abelian reduction.       

The base change conductor was further studied by the last two authors in the context of their work on \emph{motivic zeta functions} (cf. e.g. \cite{HaNi} and \cite{HaNi-book}). For an abelian variety $A$ defined over a strictly Henselian discretely valued field $K$, the motivic zeta function $Z_A(T)$ is a formal power series which encodes the asymptotic behaviour of the set of rational points of $A$ under ramified extensions of $K$. In particular, the last two authors introduced a natural decomposition
\begin{displaymath}
c(A) = c_{\tame}(A) + c_{\wild}(A)
\end{displaymath}
where both $c_{\tame}(A)$ and $c_{\wild}(A)$ are non-negative numbers. They moreover showed that if either $A$ acquires semiabelian reduction after a tame extension of $K$ (in which case $c_{\wild}(A)=0$), or if $A$ is a Jacobian, the invariant $c_{\tame}(A)$ is the unique pole of $Z_A(T)$.

\subsection{} The aim of this paper is to study the base change conductor of a \emph{Jacobian variety}, and to find effective ways to compute it from the input provided by suitable proper models of the curve over the ring of integers. The problem is local on the base, so for the rest of the introduction we let $R$ denote a complete discrete valuation ring with algebraically closed residue field $k$.

 Let $C$ be a smooth projective curve over $K$ 
 of index $1$, and $\cX$ any regular model of $C$ over $R$. Then the relative Picard scheme $ \mathrm{Pic}^0_{\cX/S} $ can be identified with the identity component of the N\'eron model of $\mathrm{Jac}(C)$ and the Lie algebra in question is dual to the Hodge bundle $ H^0(\cX , \omega_{\cX/S})$. 
To study the base change conductor for a Jacobian one can therefore equivalently investigate how the determinant of the Hodge bundle changes under semi-stable reduction of the curve $C$, see \ref{subsec:Jacobianbcc-def} for details.

At this stage, it might be useful to remark that in the complex geometric setting, the base change conductor is essentially given by the logarithm eigenvalues of the semi-simple part of the monodromy acting on $H^0(C, \omega_C)$ (cf. \cite[Theorem A]{CDG-1}, \cite[Proposition 2.10]{CDG-2}, \cite{JumpsMonodromy}).
\subsection{}
In order to study how the determinant of the Hodge bundle behaves under semi-stable reduction, we use tools from intersection theory.
One of the central ingredients in this endeavour is Deligne's Riemann--Roch theorem. This was first introduced in \cite{detcoh}, and together with further developments provides a canonical isomorphism between a power of the determinant of the Hodge bundle with certain bundles that encode local intersection numbers, modified with some further invariants of singularities  (cf. \cite{Eri16, ErikssonErratum, Saitodisc}). This theorem allows us to replace the questions about determinants of Hodge bundles with questions about local intersection theory and singularity invariants.

One of the invariants that contribute is the Artin conductor $\Art_{\cX/S}$ (cf. \eqref{def:Artinconductor}). When the total space $\cX$ is only assumed  normal and $\mathbb{Q}$-Gorenstein, one obtains further contributions from an invariant $\mu_{\cX,x}$ attached to each non-regular point $x \in \cX$. The definition of $\mu_{\cX,x}$ generalizes the formula of Laufer \cite{Laufer} of the Milnor number of an isolated surface singularity in the complex setting. We also call it the Milnor number. This invariant can, in our generality, take rational values. It seems like an interesting problem to study $ \mu_{\cX,x}$, and the related discrepancy number $\nu_{\cX, x}$, in further detail. 

 We highlight the following consequence of the above intersection theoretic intersections, which can be found in Corollary \ref{cor:basechconformula} \eqref{item1:corbasechconformula} and \eqref{item3:corbasechconformula}. Below, the sum of the invariants $\mu_{\cX,x}$ will be denoted $\mu_{\cX}$.

\begin{prop}\label{prop:introcomputation}
Suppose that $K'/K$ realises semistable reduction for $C$, that $\cX$ is a  model of $C$ over $R$ and that $\pi :\cX' \to \cX$ is a model of $C \times_K K'$ dominating $\cX \times_R R'$.

\begin{enumerate}
    \item Suppose that $\cX$ and $\cX'$ are both regular models. Then
    $$- 12 c(J)  = \frac{1}{ e }\left( \Gamma^2 + 2 \Gamma \cdot \pi^* \omega_{\cX/S} - \mathrm{Art}_{\cX'/S'} + e \cdot  \mathrm{Art}_{\cX/S}\right).$$
    \item Suppose that $\cX$ is normal and $\mathbb{Q}$-Gorenstein with at most rational singularities, and that its normalized base change $\cX'$ is smooth over $R'$. Then 
    $$        - 12 c(J)  = \frac{2 }{ e } \Gamma \cdot \pi^* \omega_{\cX/S} +   \mathrm{Art}_{\cX/S}  - \mu_\cX.
 $$
 
\end{enumerate}
 Here we have denoted by $\Gamma = \omega_{\cX'/S'} - \pi^{*}\omega_{\cX/S}$ the discrepancy divisor.
\end{prop}

These two conclusions will be the central starting points for the main applications. 

Let us also mention that in the complex geometric setting similar invariants were studied when the base is a smooth projective curve in e.g. \cite{SLTan}, \cite{SLTan2} where analogues of $c(J)$ were studied in the context of slopes.

\subsection{}

Our first application of Proposition \ref{prop:introcomputation} appears in the computation of the tame part of the base change conductor. Using tools from logarithmic geometry, the authors showed previously in \cite{logjumps} that $c_\tame(J)$ can be computed from any \emph{sncd}-model $\cX$ (see \ref{subsubsec:sncdmodel} for this terminology). More precisely, we proved that $c_\tame(J)$ equals the sum of Edixhoven's \emph{jumps}, which are non-negative rational numbers $ j_1, \ldots, j_{g(C)}$ (see \cite{EdixhovenTame} for the definition). We furthermore gave an explicit formula for each $j_i$ solely in terms of combinatorial data, by which we mean the dual graph of $\cX_k$ where each vertex is labelled with the multiplicity and the genus of the corresponding component (\cite[Thm]{logjumps}). 

As a consequence of these results, a fundamental result of Winters \cite{winters} allows us to reduce the computation of $c_\tame(J)$ to the case of equal characteristic 0. Under this assumption, both the extension $K'/K$ and the model $\cX'$ as in the first part of Proposition \ref{prop:introcomputation} can be explicitly constructed by hand, rendering all terms explicit. This gives a far more compact formula, Theorem \ref{thm:introtamebcc} below, for the tame part of the base change conductor, avoiding the complicated expressions for the individual jumps.

To state the result, we introduce some notation. We denote by $E(\cX)$ the number of nodes in the semistable curve $\cX_{k,\mathrm{red}}$. The \emph{virtual} number of nodes (cf. Definition \ref{def:virtualnodes})  is defined as
\begin{displaymath}
    R(\cX) = \frac{1}{3}   \sum_{x \in \mathrm{Sing} (\cX_{k,\mathrm{red}})} \frac{n_x^2 + n_x'^2 + (n_x,n_x')^2}{n_x n_x'}
\end{displaymath}
where $n_x$ and $n_x'$ denote the multiplicities of the components intersecting in $x$ and $(n_x, n_x')$ denotes the greatest common divisor. 

We find the following expression (cf. Corollary \ref{cor:nodediff}): 
\begin{theorem}\label{thm:introtamebcc}
    The tame part of the base change conductor is given by 

    $$c_\tame(J) =  \frac{u}{2} -\frac{1}{4}\left( R(\cX)-E(\cX) \right),$$
where $u$ is the unipotent rank of the N\'eron model of $J$.
\end{theorem}

In certain situations, one can show independently of these methods that $c_{\tame}(J) =  \frac{u}{2}$. Our formula therefore yields an interesting \emph{combinatorial} obstruction to some potential reduction types. As a (non-exhaustive) illustration, we obtain the following consequence of Theorem \ref{thm-main2} and Theorem \ref{thm-mainwild} in combination with Proposition \ref{prop-weakwildsing}, respectively:

\begin{cor}
    Suppose that $\mathrm{Jac}(C)$ \emph{either} has potential multiplicative reduction, \emph{or} acquires good ordinary reduction after a purely wild extension. Then
    $$R(\cX) = E(\cX) $$
    for any sncd-model $\cX$ of $C$.

\end{cor}

\subsection{}

To compute the wild part $c_{\wild}(J)$ of the base change conductor is more challenging. We focus on the case where $C$ has \emph{potential good reduction}, i.e., when $C$ admits a smooth model $ \cX' $ after a finite Galois extension $K'/K$. In this case the most obvious choice for the $R$-model $\cX$ is the quotient of $ \cX' $ by $G = \mathrm{Gal}(K'/K)$. This satisfies a version of the second part of Proposition \ref{prop:introcomputation}, which is the starting point for the computation of $c(J)$ in this situation. 

We assume that $K'/K$ is purely wild, in which case the quotient $\cX$ will have wild quotient surface singularities. Such singularities are very far from being understood in general, but important progress has been made in later years. We analyze in detail the case when the quotient has \emph{weak wild quotient singularities}. This class was introduced and studied in depth by Obus and Wewers \cite{ObWe} generalizing earlier work by Lorenzini \cite{Lor14} when $ G = \mathbb{Z}/p $. Further developments will appear in the upcoming Ph.D thesis of Waeterschoot \cite{weatershoot2024thesis} .

For our purposes, the key feature of a weak wild quotient singularity $ x \in \cX $ is that it is formally isomorphic to the unique singularity of a certain proper model $\cY$ of $\mathbb{P}^1$. This model is moreover obtained as a global quotient of a smooth and proper $R'$-scheme $\cZ$, by a subgroup $ H \subset \mathrm{Gal(K'/K)}$. This fact allows many local properties of a weak wild quotient singularity to be computed globally on $\cY$.

In  particular, we compute the Milnor number of a weak wild quotient singularity (Theorem \ref{Theorem:mainSwan}):
\begin{theorem}
We keep the notation above, and write $L = (K')^H$. Then 
$$ \mu_{ \cY, x } = 4\left(1 - \frac{1}{\vert H \vert} + \frac{\mathbf{sw}_{K'/L}}{\vert H \vert}\right). $$    
\end{theorem}
See \ref{subsec:ArtinSwan} and \ref{subsec-curvedata} for notation concerning the Swan conductor of an $\ell$-adic representation. Concerning the base change conductor, applying Proposition \ref{prop:introcomputation} we find the following description:

\begin{theorem}
Assume that $C$ has potential good reduction and that $\cX$ has only weak wild quotient singularities. Then    
$$ c_{\tame} = \frac{u}{2} \quad \hbox{ and }\quad c_{\wild}= \frac{1}{4} \mathrm{Sw}(C). $$

\end{theorem}

The condition in the statement of the theorem holds for instance when the special fiber $\cX'_k$ of the smooth model is an ordinary curve. The potential ordinary case, for abelian varieties in general, has been studied in depth by Chai and Kappen using sophisticated methods from rigid analytic geometry \cite{ChaiKappen}. In particular, they found a representation theoretic formula for the base change conductor of an abelian variety with potential ordinary reduction.

\subsection{}

The article is structured as follows. In Section 2 we include preliminaries and set notation for the rest of the article. In Section 3 we define the base change conductor and list a few of its fundamental properties. In Section 4 we introduce and study some local intersection theory on models of curves, including Milnor numbers and localized Chern classes. In Section 5 we recall and extend formulations of Deligne's Riemann--Roch theorem. In Section 6 we apply this theorem to provide formulas for the tame part of the base change conductor, using an explicit description of semi-stable reduction for curves with tame reduction. In Section 7 we assume that $C$ has potential good reduction, and establish, for use later in the article, several results on the Galois-quotient of the smooth model. In Section 8, we compute the base change conductor for a curve with tame potential good reduction. The last three sections are devoted to the wild case. In Section 9, we give some preliminaries on weak wild quotient singularities, and we compute the tame part of the base change conductor when these are the only singularities on the Galois-quotient. In Section 10 we compute explicitly the Milnor number of a weak wild quotient singularity. In Section 11, we compute the wild part of the base change conductor again assuming only weak wild quotient singularities occur.

\subsection{} The base change conductor can be computed in terms of quotients in general, also for semistable reduction admitting nodes, as follows from  (cf. Corollary \ref{cor:basechconformula} \eqref{item2:corbasechconformula}). Even though we decided in this paper to focus on the case of potential good reduction, most of the technical ingredients are available either directly, or after suitable modification. This is especially evident for those of local nature, such as singularity and local intersection invariants. In future work, we plan to perform a systematic study of the base change conductor for Jacobians more generally, as well as to consider various related questions.

\section{Preliminaries}
\subsection{Discretely valued fields}

\subsubsection{}
Throughout this paper, we denote by $K$ a complete discretely valued field  with ring of integers $R$ and residue field $k$. We assume that $k$ is algebraically closed with characteristic $p \geq 0$. By $\ell$ we shall always mean a prime number different from $p$. We fix a separable closure $K^s$ of $K$.

\subsubsection{} For a scheme $X$ over $K$, we write $X_{K^s}= X \times_K K^{s}.$ Similarily, for a scheme $\mathscr{X}$ over $R$, we denote by $\mathscr{X}_k $ the special fiber $\mathscr{X} \times_R k $ and by $ \mathscr{X}_K $ the generic fiber $ \mathscr{X} \times_R K $.

\subsubsection{}
Let $K'/K$ be a finite Galois extension of degree $ e = [K' \colon K]$ and denote by $R'$ the integral closure of $R$ in $K'$. Then $R'$ is again a complete discrete valuation ring, and the extension $ R \subset R' $ is ramified of index $e$ and induces an isomorphism of residue fields. 

We recall that $G = \mathrm{Gal}(K'/K)$ has the lower ramification filtration
$$ G = G_0 \supseteq G_1 \supseteq \ldots \supseteq G_i \supseteq \ldots       $$
by normal subgroups $G_i$ of $G$. The group $G_1$ is called the \emph{wild ramification subgroup} of $G$. It is a $p$-group and $G/G_1$ has cardinality prime to $p$. The group $G_i$, $i\geq 1$, is called the $i$-th ramification group.

\subsection{Artin conductor and Swan conductor}\label{subsec:ArtinSwan}

\subsubsection{} Let $V$ be a finite dimensional $\ell$-adic representation of $G$. The \emph{Swan conductor} of $V$ is defined as the sum
$$ \mathrm{Sw}(V) = \sum_{i\geq 1} \frac{1}{[G \colon G_i]} \mathrm{dim}(V/V^{G_i}).$$
The \emph{Artin conductor} of $V$ is defined as
$$ \mathrm{Art}(V) = \mathrm{dim}(V) - \mathrm{dim}(V^G) + \mathrm{Sw}(V). $$

\subsubsection{}
Let $V$ be a continuous, finite dimensional and quasi-unipotent $\ell$-adic representation of $\mathrm{Gal}(K^s/K)$. Then the action on the semi-simplification $V^{ss}$ factors through $G = \mathrm{Gal}(K'/K)$ for some finite Galois extension $K'/K$, and we set $\mathrm{Art}(V) = \mathrm{Art}(\mathrm{V^{ss}}) $ and $\mathrm{Sw}(V) = \mathrm{Sw}(\mathrm{V^{ss}}) $, respectively.

\subsubsection{} If $V$ is the representation defined by the finite field extension $K'/K$, we write for simplicity
$$ \mathrm{Sw}(V) = \mathbf{sw}_{K'/K}. $$
Simplifying further, we will frequently write $ \mathbf{sw} $ for this term if no confusion can occur.

We record the following well known description of the canonical sheaf of the finite morphism $S' = \mathrm{Spec}(R') \to S = \mathrm{Spec}(R)$, which reduces to the conductor-discriminant formula in algebraic number theory:

\begin{prop}\label{prop:canonicalofextension}
Let $s'$ be the closed point of $S'$. Then
$$ \omega_{S'/S} = \mathcal{O}_{S'}((e - 1 +  \mathbf{sw}) s'). $$
 \end{prop}

\subsection{Models of curves and abelian varieties}\label{subsec-curvedata}

\subsubsection{} We let $C$ denote a smooth projective and geometrically irreducible curve over $K$, of genus $g>0$. We shall assume that $C$ has index one, i.e., that there exists a zero cycle of degree $1$. 

A model of $C$ is a flat $R$-scheme $\cX$, endowed with an isomorphism of $K$-schemes
$$ \cX \times_R K \cong C. $$
Note that $ \cX $ is automatically integral. Unless otherwise mentioned, the scheme $ \cX $ will be assumed proper over $S$. In these cases, we define the Artin conductor of the model as the difference of $\ell$-adic Euler characteristics modified with the Swan conductor: 
\begin{equation}\label{def:Artinconductor}
    \Art_{\cX/S} = \chi(C_{K^s})-\chi(\cX_k) -  \operatorname{Sw} H^1(C_{K^s}, \mathbb{Q}_\ell).
\end{equation}
This is related to, but not the same thing as, the Artin conductor, defined in the previous section, of the virtual Galois representation $\sum (-1)^i H^i(C_{K^s}, \mathbb{Q}_\ell)$ . 
\subsubsection{}\label{subsubsec:sncdmodel}
We say that $\cX$ is an $sncd$-model if $\cX$ is regular and if the special fiber $ \cX_k = \sum_{i \in I} n_i E_i $ is a divisor with strict normal crossings on $\cX$. This is equivalent to demanding that $E_i$ is smooth for all $i$ and that $ E_i $ intersects $E_j $ transversally whenever $ i \neq j $.

\subsubsection{}  We denote by $ \Gamma(\cX) $ the dual graph of the semi-stable curve $ \cX_{k,\mathrm{red}} $. The vertex set $ \{v_i\}_{ i \in I} $ corresponds bijectively to the set $ \{E_i\}_{ i \in I} $ of irreducible components of $ \cX_k $, and whenever $ i \neq j $ there are $ \vert E_i \cap E_j \vert $ distinct edges joining $v_i$ and $v_j$.

We put $ V(\cX) = \vert I \vert $ and $ E(\cX) = \vert \mathrm{Sing} (\cX_{k,\mathrm{red}}) \vert $. Then the first Betti number of $ \Gamma(\cX) $ can be computed as
\begin{equation} \label{eq:betti1}b_1(\Gamma(\cX)) = 1 - V(\cX) + E(\cX).
\end{equation}

To each irreducible component $E_i$ one can associate the numerical data $ (n_i,g(E_i))$, where $ g(E_i) $ denotes the genus of $E_i$. By the labelled dual graph $ \widetilde{\Gamma}(\cX) $, we shall mean the data $ (\Gamma(\cX), \{(g(E_i),n_i)\}_{i \in I}) $.

\subsubsection{} 
Assume that $C$ acquires semi-stable reduction over $K'/K$. If $V_r$ denotes the semi-simplification of the $ \mathrm{Gal}(K^s/K)$-representation 
$$ \mathrm{H}^r(C_{K^s}, \mathbb{Q}_{\ell}), $$
we can also view $V_r$ as a $G$-representation. Then $ \mathrm{Sw}(V_r) $ is zero unless $r=1$, in which case we set
$$ \mathrm{Sw}(C) = \mathrm{Sw}(V_1). $$

\subsubsection{}
Assume that $C$ has genus $ g > 1 $, or is an elliptic curve. In \cite{saito-VC}, Saito proved the following cohomological criterion: $C$ has semi-stable reduction over the valuation ring $R$ of $K$ if and only if $ \mathrm{Gal}(K^s/K)$ acts unipotently on $ \mathrm{H}^1(C_{K^s}, \mathbb{Q}_{\ell}) $. Saito also proved in \emph{loc.~cit.} a tameness criterion, which can be formulated as follows: $C$ acquires semi-stable reduction over a tame extension of $K$ if and only if $\mathrm{Sw}(C) = 0 $.

\subsubsection{} Let $A$ be a semi-abelian variety over $K$ and denote by $\mathscr{A}$ its N\'eron lft-model over $R$. We shall write $\mathscr{A}^{\circ}$ for the identity component of $\mathscr{A}$. 

The identity component $\mathscr{A}_k^{\circ}$ of the special fiber is canonically an extension of an abelian variety $\mathscr{B}_k$ by the product of a torus $\mathscr{T}_k$ and a unipotent group $\mathscr{U}_k$. We define the abelian rank $a$, the toric rank $t$ and the unipotent rank of $u$ of the abelian variety $A$ to be the dimensions of the $k$-groups $\mathscr{B}_k$, $\mathscr{T}_k$ and $\mathscr{U}_k$, respectively.

\subsubsection{} Let $C$ be a smooth projective geometrically integral curve of index $1$ and let $\cX$ be a regular model of $C$. Then the N\'eron model $ \mathscr{J} $ of $J = \mathrm{Jac}(C)$ is related to $\cX$ via a natural isomorphism
$$ \mathrm{Pic}^0_{\cX/R} \cong \mathscr{J}^{\circ} $$
(cf.~\cite[9.5.4]{neron}).

\section{Definitions and basic properties }
 In this section, we recall the definition of the base change conductor and its basic properties, with a particular emphasis on the case of Jacobians.

\subsection{ Preliminaries on models }\label{subsub:comparelattices}
Let $L$ be a line bundle on a scheme $X$ over $K$. If $\mathscr{X}$ is a model of $X$ over $R$ and $\cL$ and $\cL'$ are line bundles on $\mathscr{X}$ which are isomorphic to $L$ on $X$, we can canonically write $\cL' \simeq \cL + \mathcal{O}(D)$ for a Cartier divisor $D$ supported on the special fiber. This extends to $\Q$-line bundles, and in either case we additively write that:
\begin{displaymath}
    \cL' - \cL = D \quad \hbox{ or }\quad \cL' = \cL + D
\end{displaymath}
This writing implicitly includes identifications of $\cL|_X$ and $\cL'|_X$ with $L$, but is omitted from notation. 

In the case that $\mathscr{X} = S$, any divisor is a multiple of the special point, and we simply write the multiple and omit writing the special point.

\subsection{The base change conductor}\label{subsec:basechangedef}
\subsubsection{}
Let $A$ be an abelian $K$-variety of dimension $g$ and let $\mathscr{A}$ denote its N\'eron model over $R$. Let moreover $K'/K$ be a finite separable field extension of ramification index $ e(K'/K) $. We let $\mathscr{A}'$ denote the N\'eron model of $ A \times_K K' $ over $R'$, the integral closure of $R$ in $K'$. Since $\mathscr{A}'$ is a N\'eron model, there exists a unique morphism
$$ h \colon \mathscr{A} \times_R R' \to \mathscr{A}' $$
extending the canonical isomorphism of the generic fibers. 

The base change map $h$ yields a canonical map
$$ \Omega^1_{\mathscr{A}'/R'} \to \Omega^1_{\mathscr{A}/R} \otimes_R R'. $$
Pulling back along the unit section $ e_{\mathscr{A}'}$ of $ \mathscr{A}' $, one obtains an injective homomorphism
\begin{equation}\label{eq:kappabasechange}
\kappa : \omega_{\mathscr{A}'/R'} \to \omega_{\mathscr{A}/R} \otimes_R R'
\end{equation}
of free $R'$-modules of rank $g$. We consider the following definition (\cite[2.4]{chai}):

\begin{definition} \label{Def:basechange}
Assume that $A \times_K K'$ has semi-abelian reduction over $R'$. Then we call the rational number
$$ c(A) = \frac{1}{e(K'/K)} \cdot \mathrm{length}_{R'} (\mathrm{coker}(\kappa)). $$
the \emph{base change conductor} of $A$.

\end{definition}

The definition of $c(A)$ is independent of choice of extension $K'/K$ over which $A$ has semi-abelian reduction. The most important property of $c(A)$ is that it vanishes if and only if $A$ has semi-abelian reduction over $R$ (see e.g. \cite[Prop.~4.16]{HaNi}).

\subsubsection{} 

In \cite[Ch.~6]{HaNi-book}, the base change conductor was refined as a sum of a tame and a wild part
$$ c(A) = c_{\tame}(A) + c_{\wild}(A). $$
Both terms are non-negative, and $c_{\wild}(A) $ is zero if $A$ is tamely ramified. 

If $A$ is a Jacobian, more is known  (see \cite{logjumps}). In particular, the tame part $c_{\tame}(A)$ is a rational number, and depends only on the labelled dual graph $ \widetilde{\Gamma}(\cX) $ of the minimal sncd-model $\cX$. Even in this case however, it is not known in general whether $c_{\wild}(A) = 0$ implies that $A$ is tamely ramified. It is true in every case where $c_{\wild}(A) $ has been computed.

\subsection{The Jacobian case}\label{subsec:Jacobianbcc-def} 

\subsubsection{} Let
$$ f \colon \cX \to S = \Spec(R)$$
be a regular model of $C$ with relative dualizing sheaf $\omega_{\cX/S}$. Let $ e_{\mathscr{J}} \colon S \to \mathscr{J} $ be the unit section of the N\'eron model of $J = \mathrm{Jac}(C)$. Recall that the module of invariant differentials
$$ \omega_{\mathscr{J}/S} := e_{\mathscr{J}}^* \Omega^1_{\mathscr{J}/S} $$
is a locally free sheaf on $S$ of rank equal to $g$, the relative dimension of $ \mathscr{J}/S $.

We recall the following useful identification (\cite[Prop.~2.4.4]{logjumps}).

\begin{prop}\label{prop:isocohomologyLiealgebra}
There is an $\mathcal{O}_S$-module isomorphism
$$ \alpha_{\cX} \colon \omega_{\mathscr{J}/S} \to f_* \omega_{\cX/S}. $$
\end{prop}

\subsubsection{}\label{subsubsec:Jacobiandef-bcc}
Lastly, assume that $C$ acquires semi-stable reduction after a finite separable extension $K'/K$ of ramification index $e$. Then, for any regular model $\cX'$ of $C_{K'}$ over the ring of integers $R'$, Proposition \ref{prop:isocohomologyLiealgebra} allows us to identify the map $\kappa$ in (\ref{eq:kappabasechange}), with the natural map
 $$  H^0(\cX' , \omega_{\cX'/S'}) \to H^0(\cX , \omega_{\cX/S}) \otimes_R R'. $$

By elementary lattice theory, the determinants of the two Hodge bundles above differ by $\pi^n$, where $ n = \hbox{length}_{R'} \left(\coker (\kappa)\right)$. For our purposes, this gives the most convenient way to compute the base change conductor from the input of models of curves.

\section{Some local intersection theory}

\subsection{Relative intersection theory}\label{subsec:relativeinttheory}

In this section we recall and study the  intersection theory relevant for this article. We remark that that while several results are valid in greater generality, we only announce them in the setting described in the preliminaries.

If $\cL, \cM$ are line bundles on $f: \cX \to S$, there is a canonically defined line bundle over $S$, denoted by $\langle \cL, \cM \rangle$. It is a bimultiplicative functor in line bundles on $\cX$ and commutes with base changes $S' \to S$. We refer the reader to \cite[\textsection 6-7]{detcoh} for a treatment in this generality. In this reference, one approach to these constructions is via determinants of the cohomology: 
\begin{equation}\label{def:deligneproduct}
\langle \cL, \cM \rangle = \det Rf_\ast\left(( (\cL-\mathcal{O}_{\cX}) \otimes (\cM-\mathcal{O}_{\cX}) \right) 
\end{equation}
We refer to the original article for the formalism of virtual bundles, which plays little role here except for in the proof of Proposition \ref{prop:relationtointtheory} below. Since the derived pushforward of a proper fppf morphism sends perfect complexes to perfect complexes \cite[\href{https://stacks.math.columbia.edu/tag/0B91}{0B91}]{stacks-project}, one can apply the formalism of determinants of such complexes, as developed in \cite{KMdeterminant}.

\subsubsection{}

Suppose $\cX \to S$ is a model of a curve $C \to \Spec K$, and $D$ is a Cartier divisor supported on the special fiber, and $[D]$ denotes its fundamental cycle.  There is a well-defined intersection product $D \cdot \cL := \deg \cL|_{D} \in \mathbb{Z}$ with line bundles $\cL$ on $\cX.$  

Moreover, if $\cL = O(D')$ for a Cartier divisor $D'$, supported over the special fiber, one has an equality $D \cdot \mathcal{O}(D') = D' \cdot \mathcal{O}(D),$ and we denote the symmetric product by $D \cdot D'$. The reference \cite{DelSGA7} supposes that $\cX$ is regular to identify Weil divisors and Cartier divisors, but the construction extends to Cartier divisors or even $\Q$-Cartier divisors. Given this, let us recall how these intersection numbers relate to the pairings of line bundles. 

\begin{prop}\label{prop:relationtointtheory}

    Let $\cX \to S$ be a model of $C \to \Spec K$, and suppose we are given two $\mathbb{Q}$-line bundles $\cL, \cM$,  isomorphic to line bundle $L,M$ over $C.$  Then :
    \begin{enumerate}
        \item there is a well-defined  $\Q$-line bundle $\langle \cL, \cM \rangle$, isomorphic to $\langle L, M \rangle$ over $\Spec K$. 
        \item \label{item2:telationtointtheory} If $\cL' = \cL + D$  for a  $\Q$-Cartier divisor $D$ supported on the special fiber, then 
        \begin{displaymath}
            \langle \cL', \cM \rangle = \langle \cL, \cM \rangle + (D \cdot \cM),
        \end{displaymath}
        and symmetrically in the second variable. 
    \end{enumerate}
\end{prop}

\begin{proof}
    This is essentially \cite[Proposition 4.2]{Eri16}. Since the statement is slightly more general than what the reference states, we recall the proof for the convenience of the reader. 
    
    The construction underlying the first statement is clear from the bilinearity of the pairings. 
    
    For the second point, by bilinearity we reduce to the case when $\cL$ is trivial, and $\cL' = \mathcal{O}(D)$. Then the virtual bundle, already used in the definition \ref{def:deligneproduct},
    \begin{equation}\label{def:virtualobject}
        Rf_\ast\left( (\mathcal{O}(D)-\mathcal{O}_{\cX}) \otimes (\cM-\mathcal{O}_{\cX}) \right)
    \end{equation}
    is generically trivial, and its finite length as a virtual $R$-module corresponds the number in \eqref{item2:telationtointtheory}, claimed to be $(D \cdot \cM).$ 
    
    We first prove the claim for the case when  $D$ is the special fiber $\cX_k$. In this case, the standard sequence $0 \to \mathcal{O}_{\cX} \to \mathcal{O}_{\cX}(\cX_k) \to \mathcal{O}_{\cX_k}(\cX_k)\to 0$ shows that $\mathcal{O}_{\cX}(\cX_k)-\mathcal{O}_{\cX} \simeq \mathcal{O}_{\cX_k}(\cX_k)$. Then the length of \eqref{def:virtualobject} reduces to $$\chi(\cM|_{\cX_k}) - \chi(\mathcal{O}_{\cX_k}) = \cX_k \cdot \cM .$$

     In the general case, by bilinearity one reduces to the case when  $D$ is an actual Cartier divisor. By the already treated case and bilinearity, we can suppose $D$ is effective. In this case, the rest of the argument is a repetition of the previous argument, instead relying on the equality $$\chi(\mathcal{M}|_D) - \chi(\mathcal{O}_D) = D \cdot \mathcal{M} .$$
\end{proof}

\subsection{Milnor numbers}\label{subsection:Milnor}

We continue to consider a model  $\cX \to S$ of a curve $C \to \Spec K$. We assume that $C$ is smooth, and that $\cX$ is normal and  $\mathbb{Q}$-Gorenstein. If $P\in \cX$ is a closed point, we consider an arbitrary resolution $\pi: \cX' \to \cX$. Denote by $E = \pi^{-1}(P)$ the exceptional divisor, by $\Gamma_{\cX',E}$ the dual graph of $E$, and by $g_P = \sum g_i$ the sum of the genera of the reduced part of the irreducible components of $E$. Let $V({\cX',E})$ be the number of irreducible components in $E$.  Finally, for any normal crossings resolution $(\cX', E) \to (\cX, P)$, define $b_1(\Gamma_{\cC',E})$ as its first Betti number (compare with \eqref{eq:betti1 }). The latter only depends on $(\cC, P)$, and we also denote it by $b_1(\Gamma_{\cC,P}).$ 

Also define $\Gamma_\pi$ as the discrepancy $\mathbb{Q}$-Cartier divisor of $\pi$. Finally, consider the coherent sheaf $R^1 \pi_{\ast}{\mathcal{O}_{\cX'}}$. It is an $\mathcal{O}_{\cX,P}$-module of finite length, and automatically then also an $R$-module of finite length, $p_{g,P}$.

\begin{definition}\label{def:Milnornumber}
    We define the Milnor number of $P$, in $\cX$, as:
    \begin{displaymath}
        \mu_{\cX,P} = 12 p_{g,P} + \Gamma_\pi^2 - \sum 2g_i - b_1(\Gamma_{\cC,P}) + V({\cX',E}),
    \end{displaymath}
    and the discrepancy number as 
    \begin{displaymath}
        \nu_{\cX, P} = \Gamma_\pi^2 - \sum 2g_i - b_1(\Gamma_{\cC,P}) + V({\cX',E}).
    \end{displaymath}
    
\end{definition}

The definitions are independent of the choice of the resolution $\pi$. Indeed, all expressions except  $\Gamma_\pi^2$ and $V({\cX',E})$ are independent of the resolution, and the expression $\Gamma_\pi^2 + V({\cX',E})$  is invariant under blowup in closed  regular points and hence of the choice of resolution.

\begin{lemma}\label{lemma:Milnordepend}
    The Milnor number and the discrepancy number only depend on the formal isomorphism type of $(\cX, P)$.
\end{lemma}
\begin{proof}
    The proof is standard, and we include it because of its central importance in latter sections. Moreover, the proof for the Milnor numbers  also covers the discrepancy numbers, and we hence focus on former. 
    
    Denote by $A$ the local ring $\mathcal{O}_{\cX,P}$ and by $\widehat{A}$ the completion of $A$ along the maximal ideal $\mathfrak{m}$. The map $A \to \widehat{A}$ is faithfully flat. We first address that $p_{g,P}$ only depends on the corresponding formation above $\widehat{A}.$ 
    
    Given a resolution $\cX' \to \Spec A$, 
    the base change $\cX'\times_{A} \widehat{A}$ is a resolution of $\Spec \widehat{A}$ by \cite[\href{https://stacks.math.columbia.edu/tag/0BG6}{0BG6}]{stacks-project}. Since $A \to \widehat{A}$ was faithfully flat, we find that the formation of the sheaf $R^1 \pi_{\ast} \mathcal{O}_{\cX'}$ commutes with the flat base change. The fact that $p_{g,P}$ only depends on the completion $\widehat{A}$ now follows from the following statement: Suppose we are given an $A$-module $M$, of finite $R$-length. Then $M \otimes_A \widehat{A}$ is also an $R$-module of the same $R$-finite length. This follows from the fact that  $R\to A$ is a local homomorphism of local rings, with the same residue fields. This reduces the statement to the case when $M = A/\mathfrak{m}$. Since $A/\mathfrak{m} \otimes_A \widehat{A} \simeq A/\mathfrak{m}$ the statement follows. 
    
    For the invariance of the discrepancy under the base change to the completion, write $\Gamma_{\pi} = \sum k_i E_i$ for $k_i \in \mathbb{Q}$ . Since the intersection forms on both $\cX'$ and $\cX' \times_A \widehat{A}$ are negative definite, it is determined by the equation $\Gamma_\pi \cdot E_i= 2g_i - 2 - E_i^2$, and $E_i^2$ is computed via the intersections of $E_i$ and $E_j, j\neq i$ in $E$ as in \cite[Corollaire 1.8]{DelSGA7}. Since the exceptional divisor is the same in both cases we conclude that $\Gamma_\pi$ and hence $\Gamma_\pi^2$ is invariant upon completion. 
    
    Finally, since the exceptional divisors are in fact the same before and after the completion, hence so are the genera and Betti number of the dual graph.
\end{proof}

We have the following corollary which we include for future reference:
\begin{corollary}\label{cor:milnornumber} 
    If $P \in \cX$ is either a rational double point, or formally isomorphic to a quotient arising from a finite group acting on a regular two dimensional scheme, then $g_P = b_1(\Gamma_{\cC, P}) = 0$. 
    
    In particular for such singularities, if supposed moreover rational, we have:
\begin{displaymath}
        \mu_{\cX, P} = \nu_{\cX, P} = \Gamma_\pi^2 + V({\cX',E}).
\end{displaymath}
\end{corollary}
\begin{proof}
If $x \in \cX$ would be a rational double point, it follows that for a minimal desingularization $\cX' \to \cX$, $g_P  = b_1(\Gamma_{\cX,P})=0$, cf. \cite[Section 24]{Lipman}.

 If $P \in \cX$ is formally isomorphic to a quotient singularity arising from a finite group acting on a regular two dimensional scheme, it is known that the values $g_P$ and $b_1(\Gamma_{\cX,P})$ are zero for any resolution, see e.g. \cite[Thm.~2.8]{lorenzini-wildsurfacequot}.    
\end{proof}
\subsection{Localized Chern classes}\label{subsec:localizedChern}

In this section $X \to S$ denotes a regular scheme $X$ over the spectrum $S$ of a discrete valuation ring, $R$, such that $X\to S$ is faithfully flat of constant relative dimension $d \geq 1$, and the generic fiber is smooth.

In this case, any coherent sheaf $\mathcal{F}$ is a perfect complex of finite tor dimension. Suppose $\mathcal{F}$ is locally free of rank $r$ outside of a closed subset $Z \subseteq X_k \subseteq X$, with $Z$ of finite type over $k$. Then for $i > r$ the localized Chern class $c_i^Z(\mathcal{F}) \cap [X] \in A_{d+1-i}(Z)$ is defined, as detailed below. Here $A_{j}(Z)$ denotes the Chow group of dimension $j$-cycles on the variety $Z.$

We will at most have two dimensional regular schemes, in which case coherent sheaves are resolved by a 2-term complex.  We only give, as in \cite{Bloch}, the definition in this case. So let $0 \to E^1 \stackrel{d_1}{\to} E^0 \to \mathcal{F}$ be a length 2-resolution of a coherent sheaf $\mathcal{F}$ in terms of vector bundles, $E^i$, of ranks $e_i$. Consider $p: G \to X$, where $G$ is the Grassmannian of rank $e_1$ subbundles of $E^1 \oplus E^0.$ Denote by $\xi = p^\ast E^0-\xi^1$, where $\xi^1$ is the universal subbundle of $G$ Consider the scheme theoretic image $\alpha'$ of $X \times \mathbb{A}^1$ in $G \times \mathbb{P}^1$ via the graph embedding $(x,\lambda) \mapsto ( \text{Graph}(\lambda d_1(x)), \lambda).$ Also consider the scheme theoretic image $\alpha''$ of the map $(X \setminus Z) \times \mathbb{A}^1 \to G \times \mathbb{P}^1$ determined by the map $x \mapsto \left(E^1(x) \xrightarrow{d_1}  E^0(x) \subseteq 0 \oplus E^0(x), x \right).$ With the map $i_\infty: G = G \times \{\infty \} \to G \times \mathbb{P}^1$, one can define  $\gamma = i_\infty^\ast \left(\alpha' - \alpha''\right)$ as a cycle on $G$, with support on $G_Z= G \times_X Z$. Denote by $\eta: G_Z \to Z$ the projection. Then, for $i > d$, one defines:

\begin{equation}\label{def:locChernclass}
    c_i^Z(\mathcal{F})\cap [X]  = \eta_{\ast}\left( c_{i}(\xi) \cap \gamma \right) \in A_{d+1-i}(Z).
\end{equation}

We remark that the scheme theoretic image of a reduced scheme is simply the reduced structure on the closure of its image, and since all the involved schemes are reduced the above images are in fact topological notions. 
We recall a few standard properties of these classes: 

\begin{prop}\label{prop:chernclassproperties}
    Let $\cX \to S, \mathcal{F}$ and $i$ be as above. Then the following holds for the localized Chern classes $c_i^Z(\mathcal{F}) \cap [\cX ].$

\begin{enumerate}
    \item If $Z = \bigsqcup Z_j$, and for each $j$, there is an open neighborhood $\cX_j$ of $Z_j$, which does not intersect any of the other $Z_i$, then 
    \begin{displaymath} 
        c_i^Z(\mathcal{F}) \cap [\cX ] = \sum c_i^{Z_j}(\mathcal{F}|_{\cX_j}) \cap [\cX_j ]
    \end{displaymath}
        in
        $$A_{d+1-i}(Z) = \bigoplus A_{d+1-i}(Z_j).$$
    \item If $i: Z \subseteq Z'$, then 
    \begin{displaymath}
        i_\ast  c_i^Z(\mathcal{F}) \cap [\cX ] =  c_i^{Z'}(\mathcal{F}) \cap [\cX ]
    \end{displaymath}
    \item If $Z=P$ is a closed point, the localized Chern class only depends on the localization $\mathcal{O}_{\cX ,P}$.
    \item \label{item4:chernclassproperties} If $Z=P$ is a closed point and $\mathcal{F} = \Omega_{X/S}$, the localized Chern class only depends on the completion  $\widehat{\mathcal{O}}_{\cX,P}$ of the local ring $\mathcal{O}_{\cX,P}$ along its maximal ideal.
\end{enumerate}
\end{prop}
\begin{proof}
    The first two properties are essentially in the definition of the localized Chern classes, cf. \cite[Proposition 1.1]{Bloch}. For the last two points, we first record the following general  observation. 
    
    Suppose $q: U \to X$ is a flat morphism, with the natural map $q^{-1} Z \to Z$ being an isomorphism. Since the scheme theoretic images commutes with flat base change,  \cite[\href{https://stacks.math.columbia.edu/tag/081I}{081I}]{stacks-project}, and $[q^\ast E^1 \to q^\ast E^0]=Lq^\ast \mathcal{F}\simeq q^\ast \mathcal{F}$, the formation of $\alpha', \alpha''$ commute with restriction along the maps induced by $q$. Since there is an identification of the spaces  $G_Z$ and $G_{q^{-1}(Z)}$, both $\mathcal{F}$ on $X$ and $q^\ast \mathcal{F}$ on $U$ define the same cycle $\gamma$. This is the only ingredient in \eqref{def:locChernclass} which depends on data outside of $Z.$ 

    From this we can conclude that we have the equality
    \begin{displaymath}
        c_i^Z(q^\ast \mathcal{F}) \cap [U] = c_i^Z(\mathcal{F}) \cap [\cX ] .
    \end{displaymath}

Now, if $U$ is the localization at $P$, it satisfies the above assumption, and hence $q^\ast \mathcal{F}$ only depends on $U \to S$. In the case when $\mathcal{F} = \Omega_{X/S}$ and $q: U \to X$ is the completion at $P$, we additionally use the fact that 
\begin{displaymath}
    \Omega_{A/R} \otimes_A \widehat{A} \simeq \varprojlim_n \Omega_{\widehat{A}/R}/\mathfrak{m}^n,
\end{displaymath}
where $\mathfrak{m}_P$ is the maximal ideal of $A$, see eg. \cite[Ch. 6.1, Exercise 1.3]{liu}. This latter isomorphism shows that the relative cotangent bundle only depends on the formal completion. 
\end{proof}

\section{Deligne--Riemann-Roch and base change conductors}\label{sec:DRR}

\subsection{The Deligne--Riemann--Roch isomorphism}

Consider a smooth curve $C \to \Spec K$. We continue to consider the determinant of the cohomology of $L$, the line bundle $\det R\Gamma(C, L)$ over $K$ introduced in section \ref{subsec:relativeinttheory}. We denote it by $\lambda(L)$. If  $\omega_{C}$ denotes the (relative) dualizing sheaf, as shown by Deligne in \cite{detcoh}, there is a canonical isomorphism

\begin{equation}\label{eq:genericDRR}
    12 \lambda(L) \simeq \langle \omega_{C} , \omega_{C} \rangle +  6 \langle L, L - \omega_{C} \rangle. 
\end{equation}
The isomorphism, and variations thereof, thus appearing is sometimes referred to as the Deligne--Riemann--Roch theorem. 
\subsubsection{}
Let $\cL$ be a line bundle on a model $\cX  \to S$ of $C/K$. We suppose that $\cX$  is normal and $\Q$-Gorenstein. If $\mcr L$ is a line bundle on $\cX$, generically isomorphic to $L$, both sides of $\eqref{eq:genericDRR}$ can still be defined over $S$ as $\Q$-Cartier divisors. Here we use the dualizing sheaf, or equivalently the canonical divisor, $\omega_{\cX/S}$, to extend the right hand side, via the discussion in Proposition \ref{prop:relationtointtheory}. Using the convention introduced in \ref{subsub:comparelattices}, we can write 

\begin{equation}\label{eq:DRRwithDeg}
    12 \lambda(\cL) = \langle \omega_{\cX/S}, \omega_{\cX/S} \rangle + 6 \langle \cL, \cL - \omega_{\cX/S} \rangle + \Delta,
\end{equation}
for some integer $\Delta = \Delta_{\mcr{C}/R} $. If $\mcr C$ is regular, it was proven by \cite{Saitodisc} that $\Delta = -\Art_{\cX/S}$. In \cite{Eri16, ErikssonErratum}, this was extended to normal models, which are local complete intersections over $S$, with a correction term. The correction term is given by a sum of Milnor numbers, cf. \ref{subsection:Milnor}, of the non-regular points:
\begin{displaymath}
    \mu_{\cX} = \sum_{x \in \cX_s} \mu_{\cX, x}.
\end{displaymath}

\subsection{Consequences for the Lie algebra of the Jacobian} We are in particular interested in this when we can choose the $\cX$ and $\cL$ so that 
\[\lambda(\cL) = \det \omega_{\mathscr{J}/S} = \det \left(\Lie J^0\right)^\vee ,\]
in which case questions about the base change conductor can be translated into questions about the right hand side of \eqref{eq:DRRwithDeg}.

\begin{prop}\label{thm:reftakeshi}
Let $f: \cX \to S$ be a normal  $\mathbb{Q}$-Gorenstein model of $C \to \Spec K$. Then there are canonical isomorphisms
\[12 \lambda(\cL) = \langle \omega_{\cX/S}, \omega_{\cX/S} \rangle + 6 \langle \cL, \cL - \omega_{\cX/S} \rangle   -\operatorname{Art}_{\cX/S} + \mu_{\cX} \]

and  
    \[\det \omega_{\mathscr{J}/S} = \det \left(\Lie J^0\right)^\vee \simeq \langle \omega_{\cX/S}, \omega_{\cX/S} \rangle  - \Art_{\cX/S} + \nu_{\cX}.\] 
\end{prop}

\begin{proof}
    This is essentially a reformulation of the main theorem of \cite{Eri16, ErikssonErratum}. The same proof carries over to this context, but for the sake of completeness, we include the main lines of the argument. 

    We address first the first isomorphism, by reducing to the case when $\cX$ is regular, in which case the result was known. One hence considers a resolution $\cX' \to \cX$. When comparing the left-hand sides of \eqref{eq:DRRwithDeg} the factor corresponding to $12 p_g$ in Definition \ref{def:Milnornumber} of the Milnor number appears. The factor in the same formula corresponding to $\Gamma^2$ appears when comparing the term canonical bundles $\omega_{\cX'/S}$ and $\omega_{\cX/S}$, and relating this comparison to intersection theory via Proposition \ref{prop:relationtointtheory}. Finally, the other terms appear in the comparison of topological Euler characteristics and is valid as long as $\cX$ is normal.  

    The second isomorphism follows from these considerations. Here we use the statement for $\calo_{\cX'}$ on the level of $\cX'.$ By Proposition \ref{prop:isocohomologyLiealgebra} it follows that we have $\lambda(\calo_{\cX'}) \simeq \det \omega_{\mathscr{J}/S}.$
    When rewriting the terms as described for the first isomorphism, we don't perform av comparison of $\lambda(\calo_{\cX'})$ with $\lambda(\calo_{\cX}).$ This has the effect that there is no factor $12p_g$ appearing.
\end{proof}

\begin{remark}
    One of the main application of the above proposition we have in mind is a model $\cX$ obtained as a quotient $\cX'/G$, where $\cX'$ is a regular model over the valuation ring of a Galois extension $L/K$ realizing semi-stable reduction of $C/K$, and where $G$ is the Galois group of $L/K.$
\end{remark}

\subsubsection{}\label{subsub:logregular}

Suppose now that $\cX \to S$ is a log-regular model of $C$, where we equip $\cX$ with the divisorial logarithmic structure induced from elements of $\calo_{\cX}$ which are invertible on $C$.  Let $\omega_{\cX/S}^{\log}$ be the relative logarithmic canonical bundle.  It can be defined as 
\begin{equation}\label{eq:logcanonical}
    \omega_{\cX/S}^{\log} = \omega_{\cX/S} + {\cX}_{k, red} - {\cX}_{k},
\end{equation}
It is invertible by \cite[Proposition 3.3.6]{logjumps}.
Since $\cX_{k, red}$ is $\mathbb{Q}$-Cartier, cf. \cite[Lemma 3.3.5]{BakNic}, it follows that $\omega_{\cX/S}$ itself is $\Q$-Cartier. 

Log-regularity implies that $\cX$ has only rational singularities: indeed, it is well known that the exceptional locus in the minimal resolution of singularities forms a chain of smooth rational curves, hence the argument is the same as in \cite[Prop.~3.1]{ItoSchroer}. Moreover, it follows by Grauert-Riemenschneider for normal 2-dimensional schemes \cite[\href{https://stacks.math.columbia.edu/tag/0AXD}{0AXD}]{stacks-project} that for a desingularization $\pi: \cX'\to \cX$, we have  $R^1 \pi_\ast \omega_{\cX'/S}=0.$ It then follows from \cite[Proposition 11.9]{Kollar} that $\pi_\ast \omega_{\cX'/S} = \omega_{\cX/S} $ and $\pi_\ast \calo_{\cX'} = \calo_{\cX}, R^1 \pi_\ast \calo_{\cX'} = 0$. We conclude that there are natural isomorphisms
\begin{equation}\label{eq:detcohcomparison}
      \lambda (\omega_{\cX/S}) \simeq \lambda(\omega_{\cX'/S}) \simeq \lambda(\calo_{\cX'}) \simeq \lambda(\calo_{\cX}),
\end{equation}
where we have used Grothendieck--Serre duality in the middle. The whole composition is also given by Grothendieck--Serre duality. We then have: 
\begin{displaymath}
    \lambda(\calo_{\cX}) \simeq  \det \mathrm{Lie}(\mathscr{J}/S)^\vee,
\end{displaymath}
where $\mathscr{J}$ denotes the N\'eron model of $C.$ \ 

We summarize these observations with the following proposition:
\begin{prop}\label{prop:keydetformula}
    Let $\cX \to S$ be a log-regular model of $C \to \Spec K$. Then there is a canonical isomorphism of $\Q$-Cartier divisors on $S$:
    \begin{displaymath}
        12 \det \mathrm{Lie}(\mathscr{J})^\vee \simeq \langle \omega_{\cX/S}, \omega_{\cX/S} \rangle - \Art_{{\cX}/S }+ \nu_{{\cX}}. 
    \end{displaymath}
    Here $\mathscr{J}$ denotes the N\'eron model of the Jacobian of $C. $ 
\end{prop}

Notice that since $\cX$ has rational singularities, we have $\mu_{\cX} = \nu_{\cX}$.

\subsection{A formula for the base change conductor of the Jacobian}
We summarize the above sections in the following theorem:   
\begin{theorem}\label{theorem-cond} Let $K'/K$ be a finite extension of fields, and let $\cX$ be a regular model over $R$ of a curve $C$. Let $\pi: \cX' \to \cX \times_S S'$ be a desingularization. Denote by $\Gamma = \omega_{\cX'/S'} - \pi^* \omega_{\cX/S}$ the discrepancy. Then, with $e = [K':K]$,  the following formula holds:
$$12 \left (\lambda(\calo_{\cX'}) - R'\otimes_R \lambda(\calo_{\cX}) \right ) = $$
$$ \Gamma^2 + 2 \Gamma \cdot \pi^* \omega_{\cX/S} - \mathrm{Art}_{\cX'/S'} + e \cdot \mathrm{Art}_{\cX/S}.$$
\end{theorem}
The above theorem provides a formula for the change of determinants of cohomologies, and hence an expression for the base change conductor of Jacobian in this setting. We record this as a corollary below. It is however useful to allow some singularities of the total space, and we also include some cases which affords that assumption, relying on the results of the previous section: 
\begin{corollary}\label{cor:basechconformula}
    Suppose that the extension $K'/K$ realizes semi-stable reduction for $C/K$. Then :
    \begin{enumerate}
    \item \label{item1:corbasechconformula} With the notation as in Theorem \ref{theorem-cond}, we have
    \begin{displaymath}
        - 12 c(J)  = \frac{1}{ e }\left( \Gamma^2 + 2 \Gamma \cdot \pi^* \omega_{\cX/S} - \mathrm{Art}_{\cX'/S'} + e \cdot  \mathrm{Art}_{\cX/S}\right).
    \end{displaymath}
            \item \label{item2:corbasechconformula} If $\cX$ is normal $\mathbb{Q}$-Gorenstein, and the normalized base change $\cX'$ is a regular model over $S'$, then we have
        \begin{displaymath}
        - 12 c(J)  = \frac{2 }{ e } \left(\Gamma^2+\Gamma \cdot \pi^* \omega_{\cX/S} -  \mathrm{Art}_{\cX'/S'}\right) + \mathrm{Art}_{\cX/S} -\nu_\cX.
    \end{displaymath}
    \item \label{item3:corbasechconformula} If $\cX$ is $\mathbb{Q}$-Gorenstein with at most rational singularities, and the normalized base change $\cX'$ is smooth over $S'$, then we have
        \begin{displaymath}
        - 12 c(J)  = \frac{2 }{ e } \Gamma \cdot \pi^* \omega_{\cX/S} +   \mathrm{Art}_{\cX/S} - \mu_\cX.
    \end{displaymath}

\end{enumerate}
\end{corollary}
\begin{proof}
    The first part is immediate from Theorem \ref{theorem-cond}. The second one is a consequence of the same logic as the first part, but also utilizing Proposition \ref{thm:reftakeshi}. The third is a special case of the second one and the following facts: We have $\Art_{\cX'/S'} = 0$ for a smooth model,  $\mu_{\cX} = \nu_{\cX}$ for rational singularities, and $\Gamma^2=0$ since $\Gamma$ is necessarily a multiple of the special fiber.
\end{proof}

\section{The tame part of the base change conductor}
In this section we derive formulas for the tame part of the base change conductor of the Jacobian of a curve in terms of the geometry of models of the curve. We let $\cX \to S$ be a regular sncd-model of $C$, with special fiber $\cX_k = \sum_{i \in I} n_i E_i$. 

\subsection{Preliminaries on sncd models}\label{subsec:sncd-prelim}

\subsubsection{} 
Recall that $E(\cX)$ denotes the number of nodes of $ \cX_{k,\mathrm{red}} $. We next introduce a related value, which also takes into account the multiplicites of the components of $\cX_k$. 

\begin{definition}\label{def:virtualnodes}
By the \emph{virtual number of nodes} we mean the value
$$ R(\cX) = \frac{1}{3} \cdot \sum_{x \in \mathrm{Sing} (\cX_{k,\mathrm{red}})} \frac{n_x^2 + n_x'^2 + (n_x,n_x')^2}{n_x n_x'}, $$
where $n_x$ and $n_x'$ denote the multiplicities of the two formal branches of the special fiber crossing transversally at $x$.
\end{definition}

\subsection{The tame part of the Artin conductor } 
We define the tame part of the Artin conductor of $\cX/S$ to be (compare with \eqref{def:Artinconductor})
$$ \mathrm{Art}_{\mathrm{tame}}(\cX) = \chi(\cX_{K^s}) - \chi(\cX_k). $$
It can be computed entirely in terms of the labelled dual graph $ \widetilde{\Gamma}(\cX) $. Indeed, we have
\begin{equation}\label{eq:Eulercharspecialfiber} \chi(\cX_k) = \chi(\cX_{k,red}) = \sum \chi (E_i) - E(\cX) = \sum \chi(E_i) - \sum_{i < j} E_i \cdot E_j,
\end{equation}
and (cf. \cite[Cor.~2.6.3]{JNgeometriccriteria})
$$ \chi(\cX_{K^{s}}) = \sum_i n_i \chi(E_i^{\circ}).$$
Here the intersection product is the one explained in \textsection \ref{subsec:relativeinttheory}.

\begin{lemma}\label{lemma:tameArt}
Let $u$ be unipotent rank of $J$. The following formula holds:
$$ \mathrm{Art}_{\mathrm{tame}}(\cX) = - 2u - E(\cX). $$
\end{lemma}
\begin{proof}
The toric rank $t$ equals the first Betti number of the dual graph of $ \cX $ (see \eqref{eq:betti1 }), and the abelian rank $a$ equals the sum of the genera of the irreducible components $E_i$ of the special fiber. We moreover have $g = a+t+u$. Combining with \eqref{eq:Eulercharspecialfiber} we therefore find the sequence of identities
$$ \mathrm{Art}_{\mathrm{tame}}(\cX) = \chi(\cX_{K^{s}}) - \chi(\cX_{k}) = $$
$$ (2 - 2g) - \left( \sum_{i \in \mathcal{I}} (2-2 g(E_i)) - E(\cX) \right)= $$
$$ 2(1 - V(\cX) + E(\cX)) + 2(\sum_i g(E_i) - g) - E(\cX) = $$
$$ 2(t + a - g) - E(\cX) = -2u - E(\cX). $$
\end{proof}

\subsection{The formula for the base change conductor for Jacobians}

We will exploit the second formula in Proposition \ref{thm:reftakeshi} to find a formula for the base change conductor. 
\begin{prop}\label{cor-main}
Let $C$ and $\cX/S$ be as at the start of this section, and assume in addition that $(p,n_i) = 1$ for all $ i \in I $. Then
$$ c(J) = - \frac{1}{4} \cdot (\mathrm{Art}_{\cX/S} + R(\cX)). $$
\end{prop}
\begin{proof} 
First of all, note that our assumption implies that $C$ is tamely ramified. We fix an extension $K'/K$ (of sufficiently divisible degree) that realizes semi-stable reduction for $C$. We denote the composition of normalized base change $ \widetilde{\cX} \to \cX $ and minimal desingularization $ \cX' \to \widetilde{\cX} $ by $ \pi \colon \cX' \to \cX $. Then 
$$ \pi^* \omega^{\mathrm{log}}_{\cX/S} = \omega^{\mathrm{log}}_{\cX'/S'}, $$
where one can take as definition of the logarithmic canonical bundles the expression in \eqref{eq:logcanonical}. To see this, use that the logarithmic differentials pull back to the logarithmic differentials along the pullback to the normalized base change, cf. \cite[3.3.5]{logjumps}.   Secondly, since $\widetilde{\cX}_k$ is reduced,  $\omega^{\mathrm{log}}_{\widetilde{\cX}/S'} = \omega_{\widetilde{\cX}/S'}$ and the minimal desingularization is a crepant resolution.  

From this identity, one easily derives that $\Gamma = \pi^* (\cX_{k, \mathrm{red}} - \cX_{k})$, since $\cX'/S'$ is semi-stable. It follows that
$$ \Gamma^2 = (\pi^* \cX_{k, \mathrm{red}})^2 =  e \cdot \cX_{k,\mathrm{red}}^2$$
and
$$ \Gamma \cdot \pi^* \omega_{\cX/S} = e \cdot \cX_{k, red} \cdot \omega_{\cX/S} + e \cdot \chi(\cX_{K^{s}}).$$
By adjunction $\chi(E_i) = -E_i^2 - E_i \cdot \omega_{\cX/S}$, hence
$$ \cX_{k, \mathrm{red}} \cdot \omega_{\cX/S} =  \sum_i \left( - \chi(E_i) - E_i^2 \right).$$
It follows that
\begin{eqnarray*} e^{-1} \left( \Gamma^2 + 2\Gamma \cdot \pi^* \omega_{\cX/S} \right) =\\
2 \sum_{i < j} E_i \cdot E_j - \sum_i E_i^2 - 2 \sum_i  \chi(E_i) + 2 \chi(\cX_{K^{s}}).
\end{eqnarray*}

Using that $ \cX_k = \sum n_i E_i$ has trivial intersection with any divisor in the special fiber it follows that
$$E_i^2 = - \sum_{i \neq j} \frac{n_j}{n_i} E_i \cdot E_j$$
so that
$$\sum E_i^2 = - \sum_{i < j} \frac{n_i^2 + n_j^2}{n_i n_j} E_i \cdot E_j.$$

Now we compute $\mathrm{Art}_{\cX'/S'}$. In order to do this, we use the explicit description provided in \cite{Halle-stable} of the natural map $ p \colon \widetilde{\cX} \to \cX$ and the minimal desingularization $ \rho \colon \cX' \to \widetilde{\cX} $.
For each $E_i$, let us denote by $\widetilde{E_i}^{\circ}$ the inverse image of $E_i^{\circ}$ under $p$. Then the induced map
$$ \widetilde{E_i}^{\circ} \to E_i^{\circ} $$
is \'etale of degree $n_i$, hence
$$ \chi(\widetilde{E_i}^{\circ}) = n_i \chi(E_i^{\circ}). $$

Consider now a point $ x \in E_i \cap E_j $, where $ i \neq j $. The inverse image of $ x $ under $p$ consists of $ (n_i,n_j) $ distinct points, each of them being a transversal intersection of distinct branches of $\widetilde{\cX}_k$. Moreover, the formal structure of $\widetilde{\cX}$ at any of these points is that of an $A_{n_x}$ singularity, where
$$ n_x = \frac{e \cdot (n_i,n_j)}{n_i n_j}. $$
The exceptional locus of this singularity consists of a chain of $n_x - 1 $ smooth rational curves $F_{\alpha}$. In particular, we find that $ \chi(F_{\alpha}^{\circ}) = 0 $ and that each of the $ (n_i,n_j) $ preimages of $x$ give rise to $n_x$ singular points in the special fiber of $\cX'$. By Lemma \ref{lemma:tameArt} we conclude that
$$ \mathrm{Art}_{\cX'/S'} = - e \cdot \sum_{i < j} (E_i \cdot E_j) \frac{(n_i,n_j)^2}{n_i n_j}. $$

Now we use Theorem  \ref{theorem-cond}. After dividing by $ e $, we find that
$$ - 12 \cdot c(J) = $$
$$ 2 \sum_{i < j} E_i \cdot E_j - \sum_i E_i^2 - 2 \sum_i  \chi(E_i) + 2 \chi(\cX_{K^{s}}) + $$
$$ \sum_{i < j} (E_i \cdot E_j) \frac{(n_i,n_j)^2}{n_i n_j} + \chi(\cX_{K^{s}}) - \chi(\cX_k) = $$
$$ 3 \left( \chi(\cX_{K^{s}}) - \chi(\cX_k) \right) + \sum_{i < j} (E_i \cdot E_j) \frac{n_i^2 + n_j^2 + (n_i,n_j)^2}{n_i n_j}, $$
which easily yields the desired formula.

\end{proof}

\subsubsection{} In fact, Proposition \ref{cor-main} holds for \emph{any} curve $C/K$ and \emph{any} strict normal crossings model $\cX/S$ of $C$, provided that we replace $c(\mathrm{Jac}(C))$ and $\mathrm{Art}_{\cX/S}$ by their tame counterparts $c_{\mathrm{tame}}(J)$ and $\mathrm{Art}_{\mathrm{tame}}(\cX)$. Indeed, our next result shows that the formula can be "transported" from characteristic zero.

\begin{theorem}\label{thm:tameformula}
Let $\cX/S$ be an $sncd$-model of $C$. Then
$$ c_{\mathrm{tame}}(J) = - \frac{1}{4} \cdot \left( \mathrm{Art}_{\mathrm{tame}}(\cX) + R(\cX) \right).$$
\end{theorem}
\begin{proof}
By \cite{Halle-neron} and \cite{HaNi}, the value $c_{\mathrm{tame}}(\mathrm{Jac}(C))$ only depends on the labelled dual graph $\widetilde{\Gamma}(\cX)$, and we have already observed that the same is true for $\mathrm{Art}_{\mathrm{tame}}(\cX)$.

By a result of Winters \cite{winters}, we can find a smooth geometrically connected curve $B/\mathbb{C}((t))$ and an $sncd$-model $ \mcr Y $ of $B$ such that
$$ \widetilde{\Gamma}(\cX) = \widetilde{\Gamma}(\mcr Y). $$
This implies that
$$ c_{\mathrm{tame}}(\mathrm{Jac}(C)) =  c_{\mathrm{tame}}(\mathrm{Jac}(B)) = c(\mathrm{Jac}(B)) $$
and that
$$ \mathrm{Art}_{\mathrm{tame}}(\cX) = \mathrm{Art}_{\mathrm{tame}}(\mcr Y) = \mathrm{Art}_{\mcr Y/\mathbb{C}[[t]]}. $$
Obviously, we also have that $ R(\cX) = R(\mcr Y) $. Now the result follows from Proposition \ref{cor-main}, which states that
$$ c(\mathrm{Jac}(B)) = - \frac{1}{4} \cdot \left( \mathrm{Art}_{\mcr Y/\mathbb{C}[[t]]} + R(\mcr Y) \right). $$
\end{proof}

We record also a different formulation of the above formula, which follows immediately from combining Theorem \ref{thm:tameformula} and Lemma \ref{lemma:tameArt}.

\begin{corollary}\label{cor:nodediff}
Let $C$ and $ \cX $ be as above. Then
$$ c_{\mathrm{tame}}(J) = \frac{u}{2} - \frac{1}{4} \cdot \left( R(\cX) - E(\cX) \right). $$
\end{corollary}

\begin{remark}
The term $R(\cX) - E(\cX) $ does not depend on $\cX$, but can be hard to control in general -- explicit examples show that it can take both positive and negative values. For instance, if $\cX$ is the minimal snc-model of an elliptic curve with reduction type $IV$, one computes that $R(\cX) - E(\cX) = \frac{2}{3}$, whereas for reduction type $IV^*$ the value is $- \frac{2}{3}$.

\end{remark}

\subsection{An inequality}

All the terms appearing in Theorem  \ref{thm:tameformula} have a definite sign. This is clear except possibly for $\mathrm{Art}_{\tame}(\cX)$, for which it holds that $$-\mathrm{Art}_{\tame}(\cX)\geq 0.$$ The above inequality is well-known, and follows from the fact that, in this setting, $H^1(\cX_k, \mathbb{Q}_\ell) = H^1(\cX_{K^s}, \mathbb{Q}_\ell)^{I}$, so $b_1(\cX_k) \leq b_1(\cX_{K^s})$.

Hence the formula in  question in particular entails the inequality 
\begin{equation}\label{eq:inequality1} \frac{1}{4} \left(-\mathrm{Art}_{\tame}(\cX)\right)- c_{\tame}(J)  = \frac{1}{4} R(\cX)  > 0.
\end{equation}

This type of inequalities were studied \cite{spectralgenus}, in the setting of  germs of isolated complex hypersurface singularities, determined by $f: (\mathbb{C}^2,0) \to (\mathbb{C},0).$ We suppose in the remainder of this paragraph that $ 0 \in f^{-1}(0) $ is actually a singularity. This germ can be globalized to a family of hypersurfaces in $\mathbb{P}^2$, $\cX \to (\mathbb{C},0)$,  with a single isolated singular point over the origin, so it is actually of algebraic nature. 

The inequality in question is given in terms of $\mu$, the Milnor number of $f$, and the so called spectral genus $\widetilde{p}_g$ of the  singularity. For families of hypersurfaces, one has $$\mu = -\Art_{\tame}(\cX), \widetilde{p}_g = c_{\tame}(J).$$ 
The result in \cite[Theorem B]{spectralgenus} provides us with the inequality
\begin{equation}\label{eq:spectralgenus}
    \frac{1}{6}\left(-\Art_{\tame}(\cX)\right) - c_{\tame}(J) >  0.
\end{equation}

This is of the same type as  $\eqref{eq:inequality1}$, but the latter is for an isolated singularity. In what follows, we show how it provides an inequality for the Milnor number in terms of the combinatorics  of  the resolution graphs. For this we need to compare with a normal crossings model. We pick any embedded resolution with normal crossings $\cX' \to \cX$, and denote by $E$ the exceptional divisor. Then $c_{\tame}(J)$ remains unchanged, but 
$$-\Art_{\tame}(\cX') = \mu+  e,$$
where $e :=  \#\hbox{irreducible components } E$.  As in \eqref{eq:betti1}, we  find that
$$e = E(\cX') + 1-t-n$$
where $n$ is the number of irreducible components of $\cX_k$, and $t$ is the first Betti number of the dual graph, or likewise the toric rank. 

Hence Theorem \ref{thm:tameformula} can be written as: 
\begin{equation}\label{eq:ctame}
    c_{\tame}(J) =  \frac{1}{4} \left( \mu + E(\cX') -R(\cX') + 1-t-n\right).
\end{equation}
Inserting \eqref{eq:ctame} into the inequality \eqref{eq:spectralgenus}, and using the definition of $R(\cX')$ in Definition \ref{def:virtualnodes}, we find the following: 
\begin{prop}  Suppose that $f: \cX \to \Spec \mathbb{C}[[t]]$ is a regular model of a smooth curve, with a single isolated singularity in the special fiber. With the notation as above, we have the strict inequality:

$$\mu <  \sum_{x \in \mathrm{Sing} (\cX'_{k,\mathrm{red}})} \left(\frac{n_x^2 + n_x'^2 + (n_x,n_x')^2}{n_x n_x'}-3\right) +  3 (t+n-1).$$
\end{prop}

We remark that even stronger inequalities than \eqref{eq:spectralgenus} are known for many models of curves, and are expected more generally. We refer to \cite{spectralgenus} for a further discussion. This bound seems to resist being classified as merely a formal estimate.

One could consider the same type of possible inequality over a complete DVR not necessarily of characteristic 0. The inequality in the proposition does not seem to be transportable from characteristic $0$ via a Winters' type argument as in the  proof of Theorem \ref{thm:tameformula}. It hence remains an intriguing open question whether this inequality extends to characteristic $p$.

\subsection{Chai and Yu's formula}\label{subsec-potmult}

We will now demonstrate how Chai and Yu's formula for the base change conductor of an algebraic torus \cite{chai-yu} can be used to show that the term $R(\cX) - E(\cX) $ vanishes in many cases of interest.

\subsubsection{} We denote by $ T_{\ell}(J) $ the $\ell$-adic Tate module of $J = \mathrm{Jac}(C)$, and we put $ V_{\ell}(J) = T_{\ell}(J) \otimes \mathbb{Q}_{\ell} $. The inertia group $ I = \mathrm{Gal}(K^{s}/K) $ acts on $ V_{\ell}(J) $, we write $V$ for the semi-simplification of this representation. 

\subsubsection{} 
Recall that there exists a semi-abelian $K$-variety $E$ such that $J^{an}$ is uniformized by $E^{an}$. I.e., there exists an exact sequence  
$$ 0 \to M \to E^{\mathrm{an}} \to J^{\mathrm{an}} \to 0 $$
with $M$ a $K$-lattice of rank equal to the dimension of the maximal subtorus $T$ of $E$.

We denote by $\mcr E$ the N\'eron $\mathrm{lft}$-model of $E$ . Then  $\mcr E_k^{0} $ is isomorphic to $\mcr J_k^{0}$, which implies that the abelian, toric and unipotent ranks of $J$ and $E$ coincide.

\subsubsection{}
Let us write $B = E/T$ for the abelian part of $E$. If $B$ is trivial, the main result in \cite{chai-yu} implies that
$$ c(J) = \frac{1}{4} \cdot \mathrm{Art}(V). $$
The formula continues to hold in the more general case where $B$ has good reduction over $R$ by \cite[Prop.~7.8]{chai}. Since $ \mathrm{dim}(V) - \mathrm{dim}(V^I) = 2 u$, Chai and Yu's formula can also be written as
$$ c(J) = \frac{u}{2} + \frac{1}{4} \cdot \mathrm{Sw}(V). $$

\begin{theorem}\label{thm-main2}
Assume that $ B $ has good reduction over $R$, and let $\cX/S$ be an arbitrary $sncd$-model of $C$. Then
\begin{enumerate}

\item $ c_{\tame}(J) =  \frac{u}{2} $.

\item $ E(\cX) = R(\cX) $.
\end{enumerate}
\end{theorem}
\begin{proof}
Observe first that $c(J) = c(E) $ and $ c_{\tame}(J) = c_{\tame}(E)$ by \cite[Prop.~6.2.3.2]{HaNi-book}. Moreover, by \cite[Prop.~7.1.1.3]{HaNi-book}, we have $ c_{\tame}(E) = \frac{u(E)}{2} $. To be precise, the latter result is formulated under the assumption that $E$ is a torus, but the argument generalizes immediately to the case where $B$ has good reduction over $R$. Thus, as $ u(J) = u(E) $, (1) follows. Statement (2) is immediate from (1), by Corollary \ref{cor:nodediff}.

\end{proof}

\begin{remark}
Part (2) of the above theorem gives an interesting \emph{necessary} condition for having potential purely multiplicative reduction.
\end{remark}

\section{The quotient construction}

In this section $C$ denotes a smooth projective and geometrically integral curve over $K$. We assume there exists a finite Galois extension $K'/K$ such that $ C \times_K K' $ admits a $G = \mathrm{Gal}(K'/K)$-equivariant smooth and proper model $\mathscr{Z}$ over the integral closure of $R$ in $K'$. 

Our goal is to establish some fundamental properties for the quotient map $\mathscr{Z} \to \mathscr{Y} := \mathscr{Z}/G $, that will be used frequently in later sections. 

If either $g(C) \geq 2 $ or $C$ is an elliptic curve, we always take $K'/K$ to be the \emph{minimal} extension over which $C$ acquires good reduction. However, we avoid making this an assumption, as some of the results in this section are also needed when $C \cong \mathbb{P}^1_K$ (and hence a smooth proper model exists over $S$) and $K'/K $ is non-trivial.

\subsection{Preliminaries}

We recall that $ K$ is a complete discretely valued field with ring of integers $R$ and algebraically closed residue field $k$. We fix a uniformizer $\varpi $ and denote by $p \geq 0$ the residue characteristic. We put $S = \mathrm{Spec}(R)$.

\subsubsection{} We let $K'/K$ be a finite Galois extension of degree $ e = [K':K] $ with group $G$. We denote by $R'$ the integral closure of $R$ in $K'$ and set $ S' = \mathrm{Spec}(R')$. 
\subsubsection{}
We denote by $ \varpi $ a uniformizer of $R$ and $ \varpi'$ a uniformizer of $R'$, so that $ \varpi = u (\varpi')^e $ for some unit $ u \in R' $. 

\subsubsection{}
Throughout this section, we assume that the natural action of $G$ on $ C \times_K K' $ extends to the smooth model such that $\mathscr{Z} \to S'$ is equivariant and such that $G$ acts \emph{faithfully} on $\mathscr{Z}_k$. 

\begin{remark}
    If $g(C) \geq 2$ or if $C$ is elliptic, and $K'/K$ is the "minimal" extension, the smooth model $\mathscr{Z}$ coincides with the minimal regular model and is therefore $G$-equivariant. It is moreover known that $ \mathrm{Gal}(K'/K) \hookrightarrow \mathrm{Aut}(\mathscr{Z}_k)$, see for instance \cite[Thm.~10.4.44]{liu}.

If however $C \cong \mathbb{P}^1_K$ and $K'/K$ is non-trivial, our assumption imposes a restriction in general (but it is automatically verified in our applications). 
\end{remark}

\subsection{The quotient map}

We now consider the quotient 
$$ \pi \colon \mathscr{Z} \to \mathscr{Y} := \mathscr{Z}/G. $$

\subsubsection{}
It is well known that $\mathscr{Y}$ is normal, and flat and projective over $S$. Since it is a global quotient of a regular two-dimensional scheme by a finite group, it is also $ \mathbb{Q} $-factorial (using e.g. \cite[31.17]{stacks-project}). We denote by $\omega_{\mathscr{Y}/S}$ the $\mathbb{Q}$-Cartier divisor corresponding to the canonical divisor.

In our main applications, we shall impose stronger assumptions to ensure that $\mathscr{Y}$ has \emph{rational} singularities. More precisely, this will hold, e.g., if either $K'/K$ is tame, or when $\mathscr{Y}$ has only weak wild quotient singularities.

\subsubsection{}

In the next lemma, we write $F = \mathscr{Z}_k$ and $ E = \mathscr{Y}_{k, red} $, as well as $ \eta_{F} $ and $ \eta_{E} $, respectively, for their generic points. 

\begin{lemma}\label{lemma:unram}
Let $K'/K$ and $\pi \colon \mathscr{Z} \to \mathscr{Y}$ be as above.
\begin{enumerate}

\item The homomorphism $ \mathcal{O}_{\eta_E} \to \mathcal{O}_{\eta_F} $ is unramified.

\item The quotient map induces a composition
$$ \mathscr{Z}_k \to \mathscr{Z}_k/G \to \mathscr{Y}_{k, red} $$
where the latter morphism is the normalization.

\item The special fiber $ \mathscr{Y}_k $ has multiplicity $e$ along $E$.

\end{enumerate}
\end{lemma}
\begin{proof}
We first recall two useful facts. First, because $C/K$ is geometrically integral, we have that $ K(C \times_K K') = K(C) \otimes_K K' $. Second, since $\pi \colon \mathscr{Z} \to \mathscr{Y}$ is a finite group quotient of the normal scheme $\mathscr{Z}$, we know that $\mathscr{Z}$ equals the normalization of $\mathscr{Y} \times_S S' $ in $ K(C \times_K K') $ (\cite[Lemma 2.10]{lorenzini-wildsurfacequot}).

Because $\pi \colon \mathscr{Z} \to \mathscr{Y}$ is finite, the same is true for the flat pullback
$$ \mathscr{Z} \times_{\mathscr{Y}} \mathrm{Spec}(\mathcal{O}_{\eta_E}) \to \mathrm{Spec}(\mathcal{O}_{\eta_E}). $$
In fact, we have 
$$ \mathscr{Z} \times_{\mathscr{Y}} \mathrm{Spec}(\mathcal{O}_{\eta_E}) = \mathrm{Spec}(\mathcal{O}_{\eta_F}) $$
because normalization commutes with localization, and $\eta_F$ is the only point over $\eta_E$ by smoothness of $\mathscr{Z}/S'$. This means that the natural map
$$ \mathcal{O}_{\eta_E} \to \mathcal{O}_{\eta_F} $$
is a finite local homomorphism of discrete valuation rings, hence also flat. On fraction fields we recover 
$$ K(C) \subset K(C \times_K K') $$
which is separable (since $ K \subset K' $ is separable) of degree $e$.

By flatness, also the fiber over the residue field has degree $e$, thus the extension
$$ k(E) = \mathcal{O}_{\eta_E}/\mathfrak{m}_{\eta_E} \to \mathcal{O}_{\eta_F}/\mathfrak{m}_{\eta_E} \to \mathcal{O}_{\eta_F}/\mathfrak{m}_{\eta_F} = k(F) $$
on residue fields is of degree at most $e$. On the other hand, by our assumptions $G$ acts faithfully on $k(F)$ and trivially on $k(E)$, so the above extension factors as
$$ k(E) \subset k(F)^G \subset k(F). $$
Since $k(F)^G \subset k(F)$ is Galois of degree $e$, it follows that $ k(E) = k(F)^G$. This implies that 
$$ \mathfrak{m}_{\eta_E} \mathcal{O}_{\eta_F} = \mathfrak{m}_{\eta_F}. $$
The first two assertions follow immediately.

To prove (3), we shall use the identity (of cycles)
$$ f_*f^* D = e D $$
which holds for any Cartier divisor $D$ on $\mathscr{Y}$ (see e.g. \cite[Thm.~7.2.18]{liu}). Obviously, we can write 
$$ \mathrm{div}(\varpi) = \mathscr{Y}_k = a E $$
for some integer $a$. We will show that $a=e$.  

First we compute
$$ \pi^*(aE) = \pi^*(\mathrm{div}(\varpi)) = \mathrm{div}(u (\varpi')^e) = e F, $$
with $u \in R'$ a unit. This gives
$$ ea E = \pi_* \pi^*(aE) = e [k(F):k(E)] E = e^2 E, $$
since $\pi(F) = E$ and we proved above that $e = [k(F):k(E)]$. This gives that $ e = a$.
\end{proof}

\subsubsection{The discrepancy divisor}

Let $\Gamma$ denote the ($\mathbb{Q}$-)divisor corresponding to $\omega_{\mathscr{Z}/S'} -(\pi^*\omega_{\mathscr{Y}/S}) $.  We shall refer to $\Gamma$ as the discrepancy divisor of $\pi$. 

By the results of the previous section, $\omega_{\mathscr{Z}/\mathscr{Y}}=0$ over the regular locus $\mathscr{Y}^\circ$ of $\mathscr{Y}$. The computation of $\omega_{S'/S}$ in Proposition \ref{prop:canonicalofextension} then implies the next result: 
\begin{corollary}\label{cor:discdiv}
The discrepancy divisor of $\pi$ is given by 
$$ \Gamma = (1 - e - \mathbf{sw}) \mathscr{Z}_k. $$
\end{corollary}

\subsection{The quotient map on the special fiber}

We now consider the quotient
$$ \mathscr{Z}_k \to \overline{\mathscr{Z}_k} = \mathscr{Z}_k/G. $$

\subsubsection{The Euler characteristic of $\mathscr{Y}_k$}
We next relate $\chi(\overline{\mathscr{Z}_k})$ to $ \chi(\mathscr{Y}_{k, red})  $ via the map
$$  \alpha \colon \overline{\mathscr{Z}_k} \to \mathscr{Y}_{k, red}, $$
which by Lemma \ref{lemma:unram} can also be identified with the normalization map. 

\begin{lemma}\label{lem:quotfiberprops}
Let $ \mathcal{Q} \subset \mathscr{Y}$ denote the image of the set of points in $ \mathscr{Z}_k$ with non-trivial stabilizer.

\begin{enumerate}

\item $\alpha$ restricts to an isomorphism over $\mathscr{Y}_{k, red} \setminus \mathcal{Q}$.

\item For any point $ Q \in \mathscr{Y}_{k, red} $, the pre-image $ \alpha^{-1}(Q) $ is reduced to a point.

\item $\alpha$ induces equality $\chi(\overline{\mathscr{Z}_k}) = \chi(\mathscr{Y}_k) $.
\end{enumerate}

\end{lemma}
\begin{proof}
If $P$ has trivial stabilizer, the quotient map $ \pi$ is \'etale at $P$ and hence yields an isomorphism between the completed local rings at $P$ and its image $Q$. One easily derives smoothness of $\mathscr{Y}_{k, red} $ at $Q$ from the smoothness of $\mathscr{Z}_k$ at $P$, which gives (1).

Let $Q$ be the image of $P \in \mathscr{Z}_k$.  The (reduced) preimage of $Q$ under $\pi$ consists precisely of the $G$ orbit of $P$. The same statement is true for the image $\overline{P}$ of $P$ under $\mathscr{Z}_k \to \overline{\mathscr{Z}_k} $. Assertion (2) is therefore easily obtained from considering the factorization 
$$ \mathscr{Z}_k \to \overline{\mathscr{Z}_k} \to \mathscr{Y}_{k, red}. $$

Assertion (3) follows from the previous to properties, using the scissor rule for the Euler characteristic.
\end{proof}

\subsubsection{The Artin conductor of $\mathscr{Y}$}
From Lemma \ref{lem:quotfiberprops} we immediately derive a useful statement concerning the Artin conductor of $\mathscr{Y}/S$. 

\begin{corollary}\label{Cor:ArtCond}
The Artin conductor of $\mathscr{Y}/S$ is given by
$$ \mathrm{Art}_{\mathscr{Y}/S} = \chi(\mathscr{Z}_k) - \chi(\overline{\mathscr{Z}_k}) - \mathrm{Sw}(C). $$
\end{corollary}

\subsection{A comment on a result by Lorenzini}

Lemma \ref{lemma:unram} allows us to derive an interesting structural result for the minimal resolution with strict normal crossings of $ \mathscr{Y} $. This result, which we will not use directly in the sequel, generalizes (a part of) \cite[Thm.~5.3]{Lor14}, which is formulated under the assumption that $ K'/K $ is precisely of degree $p$.

\subsubsection{}
The quotient $ \mathscr{Y} $ has finitely many isolated singularities $ Q_1, \ldots, Q_m $. We let $ \phi \colon \cX \to \mathscr{Y} $ denote the minimal resolution with strict normal crossings; by construction $\phi$ restricts to an isomorphism over $ \mathscr{Y} \setminus \cup_i Q_i $.

We let $\widetilde{E}$ denote the strict transform of $ \mathscr{Y}_{k, red} $ on $ \cX $. Then $\widetilde{E}$ intersects finitely many exceptional prime divisors $ E_1, \ldots, E_n $ non-trivially. The multiplicity of each $E_i$ in the special fiber $ \cX_k $ will be denoted $N_i$. 

\subsubsection{}
In the proof, we freely use results and notation from \cite{logjumps} and \cite{MitSme}. In particular, for any scheme $V $ flat over a discrete valuation ring, we write $V^+$ for the divisorial log structure induced by the special fiber $V_k$.

\begin{prop}
Assume that $p$ divides $e$. Then $p$ divides $N_i$ for every $i \in \{1, \ldots, n \} $. 
\end{prop}
\begin{proof}

Assume for contradiction that one of the $N_i$ is prime-to-$p$. After re-indexing, we can assume that this holds for $N_1$. Let $ Q $ be a point in the support of $ E_1 \cap \widetilde{E} $. Since $p$ does not divide $ N_1 $, the morphism $ \cX^+ \to S^+ $ of log schemes is log smooth in a Zariski neighbourhood $U$ of $ Q $. Then $ U \cap \widetilde{E} $ is an open non-empty subset of $ \widetilde{E} $, hence the generic point $ \eta_{\widetilde{E}} $ of $ \widetilde{E} $ is contained in the log smooth locus of $ \cX^+ \to S^+ $.

This implies that $ \mathscr{Y}^+ \to S^+ $ itself is log smooth after we remove finitely many closed points in the special fiber $ \mathscr{Y}_k $. We write $ \mathscr{Y}^{\circ} $ for the corresponding open subscheme of $ \mathscr{Y} $.

Since $\Gamma $ is a multiple of $ \mathscr{Z}_k $, we can compute its coefficient after replacing $ \mathscr{Y} $ by $ \mathscr{Y}^{\circ} $. Hence in the remainder of this proof we may suppose that $ \mathscr{Y}^+ \to S^+ $ is log smooth. Then $ \pi \colon \mathscr{Z} \to \mathscr{Y} $ can be identified with the underlying map of schemes for the projection map 
$$ \mathscr{Z}^+ = \mathscr{Y}^+ \times^{fs}_{S^+} (S')^+ \to \mathscr{Y}^+ $$ 
induced by $fs$-basechange along $(S')^+ \to S^+$. In particular, the formation of the log canonical sheaf commutes with pullback along $\pi$.

This now gives us identities
$$ \omega_{\mathscr{Z}/S'} = \omega_{\mathscr{Z}/S'}^{log} = \pi^*\omega_{\mathscr{Y}/S}^{log} = \pi^*\omega_{\mathscr{Y}/S} \otimes \pi^* \mathcal{O}_{\mathscr{Y}} (\mathscr{Y}_{k, red} - \mathscr{Y}_k). $$ 
On the level of divisors, and using some computations from (the proof of) Lemma \ref{lemma:unram} above, we get
$$ \Gamma = \omega_{\mathscr{Z}/S'} - \pi^*
\omega_{\mathscr{Y}/S} = (1 - e) \mathscr{Z}_k. $$
Intersecting with a section of $ \mathscr{Z} \to S'$ we obtain a contradiction, because $ \mathbf{sw} = 0 $ if and only if $ K'/K $ is tamely ramified. 

\end{proof}

\section{Curves with tame potential good reduction}\label{sec:tamegoodred}

In this section we assume that $C$ acquires good reduction after a finite tame extension $L/K $ of degree $e$. As usual, we assume that the extension $L/K$ is minimal with this property. This implies that the singularities of the quotient of the good model for the Galois action  are \emph{tame cyclic quotient singularities}. 

\subsection{Tame cyclic quotient singularities}
We start out by recalling a few facts on tame cyclic quotient singularities, and set some notation. We refer to \cite[Section 1]{Lor90} for details. 

\subsubsection{The singularities of the quotient}
We consider again the factorization 
$$ \mathscr{Z}_k \to \mathscr{Z}_k/G \to \mathscr{Y}_{k, red} $$
induced by the quotient 
$$ \pi \colon \mathscr{Z} \to \mathscr{Y} = \mathscr{Z}/G. $$
Then, since $L/K$ is tame, one can show that the second map is an isomorphism, hence $ \mathscr{Y}_{k, red} $ is regular. 

The singularities of $\mathscr{Y}$, however, are controlled by the first map. To explain this, let $Q$ be a branch point for $ \mathscr{Z}_k \to \mathscr{Z}_k/G $. We say that $Q$ is of \emph{type $d$}, with $d$ a divisor of $e$, if the Galois orbit of points in $\mathscr{Z}_k $ mapping to $Q$ has cardinality $d$. Then $d$ is necessarily a divisor of $e$, and for each $P$ in the orbit, the ramification index at $P$ equals $ e_Q = e/d $. 

A point $Q \in \mathscr{Y}$ is singular if and only if $Q$ is of type $d$ for some $ 1 \leq d < e $. It is then a tame cyclic quotient singularity with numerical data $ (e_Q, r_Q) $, where $ 0 < r_Q < e_Q $ is prime to $e_Q$ and depends only on the local action at (any) $P \in \mathscr{Z}$ mapping to $Q$.

\subsubsection{The minimal resolution}
The minimal resolution $ \rho \colon \cX_Q \to (\mathscr{Y}, Q)$ depends only on the numerical data $ (e_Q, r_Q) $. In fact, write 
$$ \frac{e_Q}{r_Q} = [a_1, \ldots, a_{l_Q}] $$
for the Jung-Hirzebruch continued fraction expansion. Then the exceptional locus of $\rho$ consists of a normal crossing chain of smooth rational curves with self-intersections $ E_1, \ldots, E_{l_Q} $ where $ - E_i^2 = a_i $. The first curve $ E_1 $ intersects the strict transform of $\mathscr{Y}_{k, red} $ transversely in a point.

\subsection{Milnor numbers}
We will now give a formula for the Milnor number $\mu_Q$ of a tame cyclic quotient singularity.
We do not include proofs, since these results are not needed elsewhere in the paper.

\subsubsection{}
Since tame cyclic quotient singularities are rational (by a similar argument as in the log-regular case in \ref{subsub:logregular}), Corollary \ref{cor:milnornumber} implies that the Milnor number of $Q$ is given by the expression
$$ \mu_Q = \Gamma_Q^2 + l_Q, $$
where $\Gamma_Q$ denotes the discrepancy divisor for the resolution $\rho$ discussed above. In order to give an expression for this divisor, we first need to recall some numerical data associated to the singularity $Q$. 

\subsubsection{}
Put $r_0 := e_Q$, $ r_1 := r_Q $ and define $ r_i $ inductively by 
$$ r_{i+1} = a_i r_i - r_{i-1}. $$
It is easily verified that 
$$ r_0 > r_1 > \ldots > r_{l_Q} = 1. $$

Next, we define universal polynomials $P_0, \ldots, P_{l_Q + 1}$ in the $a_i$-s. To do this, put $P_0 = 1$, $P_1 = 1$ and define $P_i$ inductively by 
$$ P_{i+1} = P_i \cdot a_i - P_{i-1}. $$

\subsubsection{}
Now, if we set
$$ x_i = \frac{1}{r_0}(P_i + r_i) $$
for each $i$, it is not hard to check that the discrepancy divisor is equal the $\mathbb{Q}$-divisor
$$ \Gamma_Q = \sum_{i=1}^{l_Q} (x_i - 1) E_i. $$

A direct computation then yields  that

\begin{prop}\label{prop:milnortamequot}
The Milnor number of $Q$ is given by the formula
$$ \mu_Q = 3 l_Q - \left( [a_1, \ldots, a_{l_Q}]^{-1} + \sum_{i=1}^{l_Q} a_i + [a_{l_Q}, \ldots, a_1]^{-1} \right) + 2 \left(1 - \frac{1}{e_Q}\right). $$
\end{prop}

\subsection{The main result}
We now apply the above considerations in order to obtain a formula for the base change conductor. 

\subsubsection{}
In the proof below, we shall use the notation $ \chi = \chi(\mathscr{Z}_k)$ and $ \overline{\chi} = \chi(\mathscr{Y}_{k, red})$.

\begin{theorem}\label{thm:formulapotgoodtame}
Assume that $C$ acquires good reduction after a tame extension. Then
$$ c(J) = \frac{u}{2} + \frac{1}{12} \sum_{Q \in \mathrm{Sing}(\mathscr{Y})} \left( \mu_Q - 2 \left(1 - \frac{1}{e_Q}\right) \right). $$ 

\end{theorem}
\begin{proof}
By Corollary \ref{cor:basechconformula} we have
$$ - 12 c(J) = \frac{2}{e} \Gamma \cdot \pi^* \omega_{\mathscr{Y}/S} + \left( \mathrm{Art}_{\mathscr{Y}/S} - \mu \right),  $$
where $\Gamma$ is the discrepancy divisor for $ \pi \colon \mathscr{Z} \to \mathscr{Y} $ and $\mu = \sum_{Q \in \mathrm{Sing}(\mathscr{Y})} \mu_Q $ is the sum of the Milnor numbers. By Corollary \ref{cor:discdiv} and Corollary \ref{Cor:ArtCond}, the right hand side can be written
$$ 2 \chi (1 - \frac{1}{e}) + \chi - \overline{\chi} - \mu. $$

For each divisor $ d \vert e $, let $m_d$ denote the number of branch points of type $d$.
Then we derive from the Riemann-Hurwitz formula the identity
$$ - 2 \frac{\chi}{e} = - 2 \overline{\chi} + 2 \sum_{d \vert e} m_d \left( 1 - \frac{d}{e} \right). $$

From this we obtain 
$$ - 12 c(J) = 3 (\chi - \overline{\chi}) - \sum_{d \vert e} m_d \left( \mu_d - 2 \left(1 - \frac{d}{e}\right) \right), $$
which immediately yields the asserted expression for $c(J)$.

\end{proof}

\subsubsection{}
By combining Proposition \ref{prop:milnortamequot} and Theorem \ref{thm:formulapotgoodtame}, we can also write the base change conductor as
$$ c(J) = \frac{u}{2} + \frac{1}{12} \sum_{Q \in \mathrm{Sing}(\mathscr{Y})} \widetilde{\mu}_Q, $$
where 
$$ \widetilde{\mu}_Q = 3 l_Q - \left( [a_1, \ldots, a_{l_Q}]^{-1} + \sum_{i=1}^{l_Q} a_i + [a_{l_Q}, \ldots, a_1]^{-1} \right). $$

\begin{remark}
The Milnor number $\mu_Q$, as well as $\widetilde{\mu}_Q$, can take both positive and negative values depending on the parameters $ (r_0, r_1) $ for the singularity. For instance, when $ (r_0,r_1) = (6,1)$, one computes that $ \widetilde{\mu}_Q = - \frac{10}{3}$ and $\mu_Q =  - \frac{5}{3}$, whereas for parameters $ (3,2)$ the values are $ \widetilde{\mu}_Q = \frac{2}{3}$ and $\mu_Q = 2$. We moreover have that $ \widetilde{\mu}_Q $, resp.~$\mu_Q$, vanishes for the parameters $ (2,1) $, resp.~$ (4,1) $.  
\end{remark}

\section{Curves with wild potential good reduction}
In this section, we assume that the curve $C$ acquires good reduction after a finite (minimal) extension $K'/K $ that is \emph{purely wild}. In particular, it is of degree $e = p^r$ for some integer $r > 0$. 

In order to control the singularities of the quotient $\mathscr{Y}$ of the good model $\mathscr{Z}$, we need to impose further assumptions. This will ensure that $\mathscr{Y}$ has \emph{weak wild quotient singularities}. This class of singularites was introduced, and studied in detail, by Obus and Wewers in \cite{ObWe}. 

\subsection{Weak wild quotient singularities}\label{subsec:weakwildquot}

\subsubsection{}
For any point $P \in \mathscr{Z}$, we let $ G_P \subset G $ denote the stabilizer group of $P$. If moreover $ P \in \mathscr{Z}_k$, the group $G_P$ acts on 
$$ \widehat{\mathcal{O}}_{\mathscr{Z}_k, P} \cong k[[\tau]]. $$
This action is called \emph{weakly ramified} if the higher ramification groups $ G_{P,i} $ are zero whenever $ i \geq 2 $. 

\begin{definition}
We say the action of $G$ on $\mathscr{Z}_k$ is \emph{weakly ramified} if for every closed point $P$ the action of $G_P$ on $ \widehat{\mathcal{O}}_{\mathscr{Z}_k, P} $ is weakly ramified.
\end{definition}

From now on, unless otherwise mentioned, we shall always assume that the $G$-action on $\mathscr{Z}_k$ is weakly ramified. 

\subsubsection{}
Given $ P \in \mathscr{Z}_k$ we write $K_P = (K')^{G_P}$ for the fixed field, and denote by $R_{P}$ its ring of integers. 

\begin{lemma}\label{lemma:hasweakquotsing} Assume that $G_P$ is non-trivial. Then the $\mathrm{Spec}(R_{P})$-scheme $\mathscr{Z}/G_P$ has a \emph{weak wild quotient singularity} at the image $\widetilde{Q}$ of $P$.
\end{lemma}
\begin{proof}
First of all, we observe that all three conditions in \cite[Def.~3.7]{ObWe} are verified: (a) the normalization of $\mathscr{Z}/G_P \times_{R_P} R'$ is smooth; (b) $G_P$ acts faithfully on $\mathscr{Z}_k$; (c) the action is weakly ramified at $P$. 

Strictly speaking, we also need to check that $\widetilde{Q} \in \mathscr{Z}/G_P$ is indeed singular. However, this holds due to our assumption of weak ramification. Indeed, the arguments in \cite[Section 4]{ObWe} show that $(\mathscr{Z}/G_P, \widetilde{Q} )$ is formally isomorphic to $(\mathscr{Y}_{\widetilde{Q}}, \widetilde{Q} )$, where $\mathscr{Y}_{\widetilde{Q}}$ is a normal model of $ \mathbb{P}^1_{K_P}$ obtained as a $G_P$-quotient of a smooth proper model $\mathscr{Z}_P$ of $ \mathbb{P}^1_{K'}$. The action on $\mathscr{Z}_P$ extends the arithmetic action induced by $K'$ on the generic fiber, and $P$ is the unique point with non-trivial stabilizer in the special fiber.

Assume that $\mathscr{Y}_{\widetilde{Q}}$ is regular at $\widetilde{Q}$, and hence regular everywhere. Then, since the generic fiber is $ \mathbb{P}^1_{K_P}$, we can find a rational point, which by properness extends to a section $\mathcal{S}$. By standard intersection theory $\mathcal{S}$ would intersect $(\mathscr{Y}_{\widetilde{Q}})_k$ transversely in a smooth point. But this contradicts that $(\mathscr{Y}_{\widetilde{Q}})_k$ has multiplicity $ \vert G_P \vert > 1 $.  
 
\end{proof}

The fact that our singularity is formally isomorphic to a singularity arising from the action by $G_P$ on a smooth proper model of $\mathbb{P}^1_{K'} $ will be crucial for the applications in later sections. Some other important facts concerning the singularity $ (\mathscr{Z}/G_P, \tilde{Q}) $ are listed below.

\begin{enumerate}
\item The extension $ K'/K_P $ is the \emph{unique} Galois extension which faithfully resolves $ (\mathscr{Z}/G_P, \tilde{Q}) $ \cite[Rmk.~3.8]{ObWe}.

\item The singularity $ (\mathscr{Z}/G_P, \tilde{Q}) $ is \emph{rational} \cite[Cor.~4.13]{ObWe}.

\end{enumerate}

\subsubsection{Resolution of weak wild quotient singularities}

Let $ Q \in \mathscr{Y}$ be the image of a point $P \in \mathscr{Z}_k$ with non-trivial stabilizer $G_P$. Then the  quotient map $\pi$ factors through the finite map 
$$ \mathscr{Z}/G_P \to \mathscr{Y} $$
which is \'etale at the image $\tilde{Q}$ of $P$. Hence
$$ \widehat{\mathcal{O}}_{Q} \to \widehat{\mathcal{O}}_{\tilde{Q}} $$
is an isomorphism.

Since the minimal resolution $(\mathscr{Y}, Q)' \to  \Spec(\widehat{\mathcal{O}}_{Q})$ only depends on $\widehat{\mathcal{O}}_{Q}$ (and similarly for $\tilde{Q}$)  the above isomorphism is covered by an isomorphism 
\begin{equation}\label{eq:minresiso}
    (\mathscr{Z}/G_P, \tilde{Q})' \to (\mathscr{Y}, Q)'
\end{equation}
Note that the source is the minimal resolution of a weak wild quotient singularity.

A detailed description of the minimal resolution of a weak wild quotient singularity is given in \cite[Thm.~7.8 + Figure 1]{ObWe}. We will not need the finer details in this paper, but we record some noteworthy consequences:

\begin{prop}\label{prop:elementary abelian}
If the action of $G$ on $\mathscr{Z}_k$ is weakly ramified, there exists a point with full stabilizer group $G_P = G$. In particular, the group $G$ is elementary abelian, and the curve $C$ has a rational point.
\end{prop}
\begin{proof}
It suffices to show that there exists a point $P$ with stabilizer $G_P = G$, by \cite[Prop.~2.1]{ObWe}. Assume first that $P$ is a point with stabilizer $G_P \neq G$. Then the quotient map factors as
$$ \mathscr{Z} \to \mathscr{Z}/G_P \to \mathscr{Y}, $$
where the latter map is \'etale in a neighbourhood of the image $\tilde{Q} $ of $P$. Moreover, the extension $K_P/K$ has degree $p^m$ for some $m \geq 1$. 

Using the isomorphism (\ref{eq:minresiso}) between the minimal desingularizations, we conclude that each exceptional component over $Q$ has multiplicity divisible by $p$. Hence, if no point $P$ has full stabilizer group, the special fiber $\mathscr{Y}'_k$ itself is divisible by $p$ (recall that we already know that this holds for $\mathscr{Y}_k$). But this contradicts our assumption that $C$ has index one.

Let therefore $ P $ be a point with $G_P = G$. Then \cite[Prop.~2.1]{ObWe} implies that $G$ is elementary abelian. By \cite[Cor.~7.16]{ObWe} we can also conclude that in fact $ C(K) \neq \emptyset $.

\end{proof}

\subsection{Computation of $c_{\mathrm{tame}}(J)$ }

We let $ \mathcal{Y}' \to \mathcal{Y} $ denote the minimal resolution of singularities. It is an isomorphism outside of $ \mathrm{Sing}(\mathcal{Y}) = \{ Q_1, \ldots, Q_m \}$. We also know that $ \mathcal{Y}_{k, red} $ is irreducible, and smooth away from the $Q_i$-s. In particular $ \mathcal{Y}' $ is an sncd-model of $C$, so the number of nodes $ E(\mathcal{Y}') $ as well as the virtual number of nodes $ R(\mathcal{Y}') $ introduced in Subsection \ref{subsec:sncd-prelim} are defined.

\begin{theorem}\label{thm:wildquotctame}
For the minimal resolution $ \mathcal{Y}' $ of $ \mathcal{Y} $, the equality 
$$ R(\mathcal{Y}') = E(\mathcal{Y}') $$
holds. In particular, 
$$ c_{\mathrm{tame}}(J) = \frac{u}{2}. $$ 
\end{theorem}
\begin{proof}
The expressions $ R(\mathcal{Y}') $ and $ E(\mathcal{Y}') $ are in an obvious way sums of local expressions $R_{Q}$ and $E_{Q}$ associated with the minimal resolution 
$$ (\mathscr{Y}, Q)' \to  \Spec(\widehat{\mathcal{O}}_{Q}) $$
at each singular point $Q$. It therefore suffices to show that $ R_{Q} = E_{Q} $. 

The isomorphism $(\mathscr{Z}/G_P, \tilde{Q})' \to (\mathscr{Y}, Q)'$ to instead compute $R_{\tilde{Q}}$ and $E_{\tilde{Q}}$. To be clear, $(\mathscr{Z}/G_P, \tilde{Q})'$ is an $R_P$-scheme, so the multiplicities in the special fibers of the source and target differ by a factor $ \vert G/G_P \vert $. But an elementary computation shows that 
this factor cancels out and we find $R_{\tilde{Q}} = R_{Q}$. 

We can therefore replace $(\mathscr{Z}/G_P, \tilde{Q})$ by $(\mathscr{Y}_{\tilde{Q}}, \tilde{Q})$, where we use the notation introduced in the proof of Lemma \ref{lemma:hasweakquotsing}. Recall that $ \mathscr{Y}_{\tilde{Q}} $ is a normal proper $R$-model of $ \mathbb{P}^1_{K_P}$ with $ \tilde{Q} $ the unique singular point (and $(\mathcal{Y}_{\tilde{Q}})_{k, red}$ is smooth outside $ \tilde{Q} $ as well). Then the minimal resolution $\mathscr{Y}_{\tilde{Q}}' \to \mathscr{Y}_{\tilde{Q}} $ yields a a proper regular sncd-model of $ \mathbb{P}^1_{K_P} $. We claim that this fact implies that $R_{\tilde{Q}} = E_{\tilde{Q}} $.

Indeed, we can find a successive contraction 
$$ \rho \colon \mathscr{Y}_{\tilde{Q}}' \to ( \mathscr{Y}_{\tilde{Q}})_{\mathrm{min}} $$
of $(-1)$-curves intersecting the other components in at most two points, with $(\mathscr{Y}_{\tilde{Q}})_{\mathrm{min}}$ a relatively minimal model, and its (irreducible) special fiber is reduced and isomorphic to $ \mathbb{P}^1_k $. 
The identity $ R((\mathcal{Y}_{\tilde{Q}} )_{\mathrm{min}}) = E((\mathcal{Y}_{\tilde{Q}})_{\mathrm{min}}) $ therefore trivially holds. Since $ \rho $ is a composition of blow-ups in closed points in the special fiber, the analogous identity holds also for $\mathcal{Y}_{\tilde{Q}} $, since the difference 
$$ R(\mathcal{Y}_{\tilde{Q}}) - E(\mathcal{Y}_{\tilde{Q}}) $$
is independent of the choice of sncd-model.

The second statement now follows from Corollary \ref{cor:nodediff}.

\end{proof}

\subsection{Potential ordinary reduction}
In this subsection we keep the assumptions at the beginning of this section, but we do \emph{not} assume that the $G$-action on $\mathcal{Z}_k$ is weakly ramified. Rather, we shall prove that this property holds when we assume that $\mathcal{Z}_k$ is \emph{ordinary}, \emph{i.e.}, when $\Jac(\mathcal{Z}_k)$ has maximal $p$-rank.

\subsubsection{Riemann-Hurwitz formula}\label{subsubsec:RH}

For simplicity, we write $D = \mathcal{Z}_k$ and $ \overline{D} = \mathcal{Z}_k/G$, and consider the quotient map
$$ \phi \colon D \to \overline{D}. $$
For any (closed) point $ P \in D $, the stabilizer $G_P$ acts on $ \hat{\mathcal{O}}_{D,P} \cong k[[x]]$, with associated filtration of higher ramification groups
$$ G_P = G_{P,0} = G_{P, 1} \supseteq \ldots \supseteq G_{P, i} \supseteq \ldots $$

We put
$$ d_P = \sum_{i \geq 0} (\vert G_{P,i} \vert - 1)$$
Since $G_P$ is a $p$-group, we can write
$$ d_P = 2 \vert G_P \vert - 2 + \epsilon_P, $$
where $\epsilon_P \geq 0 $. In this notation, the Riemann-Hurwitz  formula can be written
\begin{equation}\label{eq:RH} 2g - 2 = \vert G \vert (2 \overline{g} - 2) + \sum_{P \in D} d_P. 
\end{equation}

Clearly, the action is weakly ramified at $ P $ if and only if $ \epsilon_P = 0$. If this holds for each point, i.e.,  the action of $G$ on $D$ is weakly ramified, the formula \eqref{eq:RH} takes the simpler form
\begin{equation}\label{eq:RH2}
    2g - 2 = \vert G \vert (2 \overline{g} - 2) + \sum_{P \in E} \left(2 \vert G_P \vert - 2\right). 
\end{equation} 

\subsubsection{Deuring-Shafarevich formula}\label{subsubsec:DS}

Let $ O_P $ be the orbit of $P$. Then $ \vert G \vert = \vert G_P \vert \cdot \vert O_P \vert $ by the orbit-stabilizer theorem. Moreover, for any $P' \in O_P$, it is elementary that $ \vert G_P \vert = \vert G_{P'} \vert $.

We denote by $ \gamma$ the $p$-rank of (the Jacobian of) $D$, and by $ \overline{\gamma} $ the $p$-rank of $\overline{D}$. In this notation, the Deuring-Shafarevich   formula states that
\begin{equation}\label{eq:DS}\gamma - 1 = \vert G \vert ( \overline{\gamma} - 1) + \sum_{i} (\vert G \vert - \vert O_i \vert), 
\end{equation}
where the sum runs over the \emph{small} orbits $O_i$ (i.e., $\vert G \vert - \vert O_i \vert > 0$).

\subsubsection{Potential ordinary reduction}
The next lemma generalizes a result by Lorenzini \cite[Subs.~2.2]{Lor14}.

\begin{lemma}\label{lemma-nohigherram}
Assume that $D$ is \emph{ordinary}, \emph{i.e.}, that  $g = \gamma $. Then $ \epsilon_P = 0 $ for all points $ P \in D $.
\end{lemma}
\begin{proof}
Let $ O_i $ be a small orbit. Since $ \vert G_P \vert = \vert G(i) \vert $ for every $ P \in O_i $, we can write
$$ \sum_{P \in O_i} d_P = (2 \vert G(i) \vert - 2 ) \vert O_i \vert + \epsilon_i$$
for some $ \epsilon_i \geq 0 $. Combining the formulas \eqref{eq:RH2} and \eqref{eq:DS} , we therefore find (since $ \vert G \vert  = \vert G(i) \vert \cdot \vert O_i \vert $) that
$$ \vert G \vert (2 \overline{g} - 2) + 2 \sum_{i} (\vert G \vert - \vert O_i \vert) + \sum_{i} \epsilon_i = \vert G \vert ( 2 \overline{\gamma} - 2) + 2 \sum_{i} (\vert G \vert - \vert O_i \vert). $$

In conclusion $ \epsilon_i = 0 $ for every $i$, and hence $ \epsilon_P = 0 $ for every $ p \in D $. 
\end{proof}

As an immediate consequence of Lemma \ref{lemma-nohigherram}, we see that a curve that acquires potential good ordinary reduction after a purely wild (minimal) extension provides an example of weakly ramified action. 

\begin{prop}\label{prop-weakwildsing}
If $\mathcal{Z}_k$ is ordinary, the $G$-action on $\mathcal{Z}_k$ is weakly ramified.   
\end{prop}

Explicit examples are easy to produce for elliptic curves, see \cite[Thm.~4.2]{lorenzini-wildsurfacequot} for a discussion. For the higher genus case, see \cite[Thm.~6.8]{Lor14} for a detailed treatment when $ [K':K] = p$.

\section{Milnor numbers of weak wild quotient singularities}
In this section, we keep the assumptions and notation from \ref{subsec:weakwildquot}. In particular, the extension $K'/K$ is purely wild, and the Galois action on the special fiber of the good model $\mathscr{Z}$ is weakly ramified. 

In Theorem \ref{Theorem:mainSwan} below we explicitly compute the Milnor number of a weak wild quotient singularity. Since these are the only singularities that can occur on the quotient $ \mathscr{Y} $, we obtain an expression for the term $ \mu = \sum_{Q \in \mathscr{Y} } \mu_Q $ appearing in the formula for the base change conductor in Corollary \ref{cor:basechconformula}.

\subsection{Computation of Milnor numbers}
Let $ Q \in \mathscr{Y} $ be the image of a point $ P \in \mathscr{Z} $ with non-trivial stabilizer $G_P$.
By Lemma \ref{lemma:Milnordepend} the Milnor number $\mu_Q$ only depends on $ \widehat{\mathcal{O}}_{\mathscr{Y}, Q}$. For the purpose of computing $\mu_Q$, we can therefore replace $\mathscr{Y}$ with the normal proper model $\mathscr{Y}_{\widetilde{Q}}$ (notation as in Theorem \ref{thm:wildquotctame}) of $\mathbb{P}^1_{K_P}$, which is singular only at $ \widetilde{Q} $.  

\subsubsection{}
We next present an explicit formula for $\mu_{ \widetilde{Q} }$. We use the notation $  e_P = [K' : K_P]$ and $S_P = \mathrm{Spec}(R_P)$.

\begin{theorem}\label{Theorem:mainSwan}
The Milnor number of $ \widetilde{Q} $ is
$$ \mu_{ \widetilde{Q} } = 4\left(1 - \frac{1}{e_P} + \frac{\mathbf{sw}_{K'/K_P}}{e_P}\right). $$
\end{theorem}
\begin{proof}
Write $ J_{ \widetilde{Q} } $ for the Jacobian of $\mathbb{P}^1_{K_P}$. Then $c(J_{ \widetilde{Q} })=0$ and the formula in Corollary \ref{cor:basechconformula} takes the form 
$$ 0 = \frac{2}{e_P} \Gamma \cdot \pi^* K_{\mathscr{Y}_{ \widetilde{Q} }/S_P} + (\mathrm{Art}_{\mathscr{Y}_{ \widetilde{Q} }/S_P} - \mu_{ \widetilde{Q} }). $$

Using Corollary \ref{cor:discdiv}, we compute
$$  \frac{2}{e_P} \Gamma \cdot \pi^* K_{\mathscr{Y}_{ \widetilde{Q} }/S_P} = 4 \left( 1 - \frac{1}{e_P} +  \frac{\mathbf{sw}_{K'/K_P}}{e_P} \right). $$
We moreover have that $ \mathrm{Art}_{\mathscr{Y}_{ \widetilde{Q} }/S_P} = 0 $,  since $(\mathscr{Y}_{ \widetilde{Q} })_{k, red}$ is rational. 
\end{proof}

\begin{remark}
The above formula for the Milnor number can also be established more directly, without Corollary \ref{cor:basechconformula}. In fact, an alternative proof can be obtained by combining certain fundamental properties of the Deligne bracket (behaviour under basechange and birational maps, respectively) with standard results on proper and normal models of $\mathbb{P}^1$.
\end{remark}

\subsubsection{}

Now we return to the global situation $ \pi \colon \mathscr{Z} \to \mathscr{Y}$.
\begin{definition}\label{def:pointtypei}
 We say that a singular point $Q$ is of type $i$ if $\vert G_P \vert = p^{r-i} $ for every $P \in \mathscr{Z} $ mapping to $Q$.    
\end{definition}

For each $i \in \{0, \ldots, r-1\}$, let $ Q_{i,j} $ be the points of type $i$, where $ 1 \leq j \leq m_i $. If $ P \mapsto Q_{i,j} $, the extension $K' \supset K_P $ depends only on $Q_{i,j} $ as $G$ is abelian by Proposition \ref{prop:elementary abelian}, so the stabilizers coincide for all points in the same $G$-orbit. In particular, the corresponding Swan conductor $\mathrm{Sw}_{K'/K_P}$ will be denoted $\mathbf{sw}_{Q_{i,j}}  $. 

In this notation, Theorem \ref{Theorem:mainSwan} therefore yields:

\begin{prop}\label{Prop:Swansum}
The sum of all Milnor numbers of the singularities of $\mathscr{Y}$ is
$$ \mu = 4 \sum_{i=0}^{r-1} m_i \left(1 - \frac{1}{p^{r-i}}\right) + 4 \sum_{i=0}^{r-1} \sum_{j=1}^{m_i} \frac{\mathbf{sw}_{Q_{i,j}}}{p^{r-i}}. $$

\end{prop}

\subsubsection{}
Our strategy in the weak wild case can be pursued for computing $\mu_Q$ and $\nu_Q$ also for other types of wild quotient singularities that are formally isomorphic to the unique singularity of a proper normal model. We provide below an example in the setting of elliptic curves. 

\begin{example}
 In this example, $p = 3$, and $E$ will be an elliptic curve with additive reduction that acquires good reduction after an extension $K'/K$ of degree $3$. Then $ G = \mathbb{Z}/3 $ acts on the good model $\mathcal{E}'$ of $E \times_K K'$. Since $p=3$ and $(\mathcal{E}')_k$ is an elliptic curve with an order $3$ automorphism, it has $j$-invariant $0$ and there is a unique fixed point, denoted $P$.

The quotient $\mathcal{E}$ is regular except for (possibly) at the image $Q$ of $P$. Note that $ \mathcal{E}_{k, red}$ is necessarily rational since $E$ has additive reduction, so consequently
$$ \mathrm{Art}_{\mathcal{E}/R} = - 2 - \mathrm{Sw}(E). $$  
We, therefore, obtain from Proposition \ref{thm:reftakeshi} the formula 
$$ c(E) = \frac{1}{6} + \frac{1}{12} \mathrm{Sw}(E) + \frac{1}{12} \nu_Q. $$ 

For elliptic curves the base change conductor has been computed in full generality (see \cite[7.2]{HaNi-book}). 
Since $E$ has wild potential good reduction, and $p=3$, it is known that $ c_{\wild}(E) = \frac{1}{12} \mathrm{Sw}(E)$. Thus we find
$$ \nu_Q =  12 c_{\tame}(E) - 2. $$  

The exact value of $c_{\tame}(E)$ depends on the reduction type of the minimal regular model of $E$. Since $E$ is wildly ramified, only $II$, $II^*$, $ IV$ or $IV^*$ in the Kodaira classification could occur, in which case $\nu_Q$ would have value $0$, $8$, $ 2 $ or $ 6$, respectively.

\end{example}

\subsection{The discrepancy divisor}

It is a remarkable fact that in our computation of $\mu_Q$ in Theorem \ref{Theorem:mainSwan}, we never needed to explicitly know the discrepancy divisor $\Gamma_Q$. However, for other applications, it might be desirable to have an expression for $\Gamma_Q$. For the sake of context, we discuss what happens in the $p$-cyclic case. We do not include proofs, as these results are not needed later in the paper.

\subsubsection{}
Assume that $ e = p $, and let $ Q $ be a weak wild quotient singularity. Then, in the terminology of \cite{ObWe}, $Q$ is of type $(r,s)$ and the minimal resolution of $Q$ can be described as follows. There is a unique node, \emph{i.e.} a vertex of valency $3$, with three chains emanating from it: 
\begin{enumerate} 
\item There is one chain of ($-2$)-curves $D_j$ for $ j \in \{2, \ldots, sp \} $. 
\item There is one chain of curves, denoted $ E_i $, of cyclic quotient singularity type, corresponding to the numerical data $ r_0 = p $, a certain $ 0 < r_1 < r_0 $ and 
$$ \frac{r_0}{r_1} = [a_1, \ldots, a_l]. $$
We denote by $P_i$ the universal polynomials introduced in Section \ref{sec:tamegoodred}. 
\item The last chain of curves, denoted $E_{-i}$ is also of cyclic quotient singularity type, corresponding to the numerical data $ r_0 = p $, a certain $ 0 < r_{-1} < r_0 $ and 
$$ \frac{r_0}{r_{-1}} = [a_{-1}, \ldots, a_{-m}]. $$
We denote by $P_{-i}$ the universal polynomials for this chain.
\end{enumerate}
The unique node will be labeled $ D_1 = E_0 $. The three vertices at the end of the chains are denoted $ D_{ps} $, $E_l$, and $ E_{-m}$: 
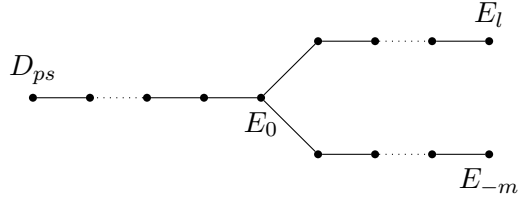
\begin{center}
\begin{tikzpicture}[scale=1] 
    \tikzset{dot/.style={circle,fill,minimum size=3pt,inner sep=0pt,outer sep=0pt}}

    \node[dot, label=below:$E_0$] (center) at (0,0) {};

    \node[dot] (downright1) at (0.75,-0.75) {}; 
    \node[dot] (downright2) at (1.5,-0.75) {};
    \node[dot] (downright3) at (2.25,-0.75) {};
    \node[dot, label=below:$E_{-m}$] (downright4) at (3,-0.75) {};
    \node[dot] (upright1) at (0.75,0.75) {};
    \node[dot] (upright2) at (1.5,0.75) {};
    \node[dot] (upright3) at (2.25,0.75) {};
    \node[dot, label=above:$E_l$] (upright4) at (3,0.75) {};
    \node[dot] (left1) at (-0.75,0) {};
    \node[dot] (left2) at (-1.5,0) {};
    \node[dot] (left3) at (-2.25,0) {};
    \node[dot, label=above:$D_{ps}$] (left4) at (-3,0) {};

    \draw (center) -- (downright1) -- (downright2);
    \draw (center) -- (upright1) -- (upright2);
    \draw (center) -- (left2);
    \draw (left3) -- (left4);

    \draw (downright3) -- (downright4) ;
    \draw (upright3) -- (upright4) ;

    \draw[dotted] (downright2) -- +(0.75,0);
    \draw[dotted] (upright2) -- +(0.75,0);
    \draw[dotted] (left2) -- +(-0.75,0);
\end{tikzpicture}
\end{center}
\subsubsection{}
Now we give the formula for $\Gamma_Q$. To simplify notation, write $ \alpha = ps $ and $ \lambda = p (1 - \alpha) + 2 \alpha $. 

Define rational numbers
$$ x_i = \frac{p P_i + \lambda r_i}{p^2} $$
for each $ i = -m, \ldots, 0, \ldots, l $. 
One easily checks that $ x_0 = \lambda/p $. We furthermore put $ k_i = x_i - 1$ and $ y_j = \frac{\alpha - j + 1}{\alpha} k_0 $.

\begin{prop}
The discrepancy divisor of $Q$ is given by the formula
$$ \Gamma_Q = \sum_{i=-m}^l k_i E_i + \sum_{j=2}^{\alpha} y_j D_j. $$
\end{prop}

\subsubsection{}

In fact the value $s$ above has a deeper meaning; it is the unique ramification jump $ s = s_{K'/K} $ of $ K'/K $. This means that $\mathrm{Sw}_{K'/K} = (s-1)(p-1) $, and we therefore obtain
$$ \mu_Q = 4s_{K'/K}\left(1 - \frac{1}{p}\right). $$

\subsubsection{Some concluding remarks}
One can in principle compute $\Gamma_Q$ explicitly for any weak wild quotient singularity,  using the detailed information provided in \cite{ObWe}. We have not attempted to do this beyond the $p$-cyclic case. 

It is interesting to compare the expression of $\mu_Q$ in the tame cyclic case and the weak wild case. In the latter, the self-intersection numbers disappear completely, and one is left with a compact expression. In view of the formula for $\Gamma_Q$ above, this might at first sight seem curious, but is due to the fact that the two chains indexed by ii) and iii) in the discussion above are, in some sense, \emph{dual}, which causes cancellation (the essential identity relating the two sets of numerical data can be found in \cite[(3.10)]{KiVi}).

\section{The base change conductor in the weakly wild case}


We keep the notation and assumptions from the previous section. In particular, we assume that $C$ acquires good reduction after a purely wild extension $K'/K$, and we denote by $\mathscr{Z}$ the good model over $R'$. To simplify the notation in this section, we shall write $ \chi = \chi(\mathscr{Z}_k)$ and $ \overline{\chi} = \chi(\mathscr{Y}_{k, red})$. We moreover write $ \mathrm{Sw}(C) $ for $ \mathrm{Sw} (\mathrm{H}^1(C_{K^s}, \mathbb{Q}_{\ell})) $ and $\mathbf{sw}$ for $\mathbf{sw}_{K'/K}$. 

\subsection{The main result}
The following provides a closed formula for the base change conductor in terms of well-known invariants: 
\begin{theorem}\label{thm-mainwild}
Assume that the action of $G = \mathrm{Gal}(K'/K) $ on $\mathscr{Z}_k$ is weakly ramified. Then
$$ c_{\tame} = \frac{u}{2} \quad \hbox{ and }\quad c_{\wild}= \frac{1}{4} \mathrm{Sw}(C). $$
\end{theorem}
\begin{proof}
Corollary \ref{cor:basechconformula} states that
$$ - 12 c(J) = \frac{2}{e} \Gamma \cdot \pi^* K_{\mathscr{Y}/S} + (\mathrm{Art}_{\mathscr{Y}/S} - \mu) $$
where $e$ as usual denotes the degree of $K'/K$. The discrepancy divisor $\Gamma$ was computed in Corollary \ref{cor:discdiv}, the Artin conductor $\mathrm{Art}_{\mathscr{Y}/S}$ was computed in Corollary \ref{Cor:ArtCond} and in Proposition \ref{Prop:Swansum} we computed the sum $\mu$ of the Milnor numbers of $\mathscr{Y}$. Substituting accordingly, we find
$$ - 12 c(J) = 2 \chi \left(1 - \frac{1}{e}\right) + 2 \frac{\chi}{e} \mathbf{sw} $$
$$   + \chi - \overline{\chi} - \mathrm{Sw}(C) $$
$$ - \left( 4 \sum_{i=0}^{r-1} m_i \left(1 - \frac{1}{p^{r-i}}\right) + 4 \sum_{i=0}^{r-1} \sum_{j=1}^{m_i} \frac{\mathbf{sw}_{Q_{i,j}}}{p^{r-i}} \right). $$
By the Riemann-Hurwitz formula (see \ref{subsubsec:RH}), this can be written 
$$ - 12 c(J) = 2 \chi + \left( - 2 \overline{\chi}  + 4 \sum_{i=0}^{r-1} m_i \left(1 - \frac{1}{p^{r-i}}\right)\right) + 2 \frac{\chi}{e} \mathbf{sw} $$
$$   + \chi - \overline{\chi} - \mathrm{Sw}(C) $$
$$ - \left( 4 \sum_{i=0}^{r-1} m_i \left(1 - \frac{1}{p^{r-i}}\right) + 4 \sum_{i=0}^{r-1} \sum_{j=1}^{m_i} \frac{\mathbf{sw}_{Q_{i,j}}}{p^{r-i}} \right). $$

The above expression simplifies to
$$ - 12 c(J) = 3 (\chi  - \overline{\chi}) $$
$$ + 2 \frac{\chi}{e} \mathbf{sw} - 4 \sum_{i=0}^{r-1} \sum_{j=1}^{m_i} \frac{\mathbf{sw}_{Q_{i,j}}}{p^{r-i}}  - \mathrm{Sw}(C). $$

Applying again the Riemann-Hurwitz formula, we can therefore write
$$ - 12 c(J) = -6 u $$
$$ + \left(2 \overline{\chi} - 4 \sum_{i=0}^{r-1} m_i \left(1 - \frac{1}{p^{r-i}}\right) \right)\mathbf{sw} - 4 \sum_{i=0}^{r-1} \sum_{j=1}^{m_i} \frac{\mathbf{sw}_{Q_{i,j}}}{p^{r-i}}  - \mathrm{Sw}(C). $$
Here we have, suggestively, split out the "wild part" on the last line. Therefore, to complete the proof we can apply Proposition \ref{Prop:KeyIden} below.
\end{proof}

\begin{prop}\label{Prop:KeyIden}
The following identity holds:
$$\left( \overline{\chi} - 2 \sum_{i=0}^{r-1} m_i \left(1 - \frac{1}{p^{r-i}}\right) \right)\mathbf{sw} - 2 \sum_{i=0}^{r-1} \sum_{j=1}^{m_i} \frac{\mathbf{sw}_{Q_{i,j}}}{p^{r-i}} = - \mathrm{Sw}(C). $$
\end{prop}

The proof of Proposition \ref{Prop:KeyIden} we shall present uses the theory of localized Chern classes as developed in \cite{Bloch} and \cite{KatoSaito} in the context of Bloch's conductor conjecture. We believe this approach is most conceptual, and moreover that some of the arguments could be useful in other settings. However, it is also possible to prove the identity in a more direct way, through (somewhat lengthy) explicit computations. We provide a brief explanation at the end.

\subsection{Preliminaries for the proof}

Let $ P $ be a closed point in the special fiber of $ \mathscr{Z} $. Then we have a factorization $ S' \to S_P \to S $ corresponding to the fixed field $ (K')^{G_P}$ of the stabilizer group. If $ W $ is an $S_P$-scheme we will, when the base scheme is clear from the context, write $ B(W) $ for $W \times_{S_P} S' $ and $ NB(W) $ for the normalization of $B(W)$.

\subsubsection{The geometric construction} 

Let $ \varphi \colon \mathscr{Y}' \to \mathscr{Y} $ be the minimal desingularization. Let moreover
$$ \mathscr{Z}' \to NB(\mathscr{Y}') $$
be the minimal resolution with strict normal crossings. Then $ \mathscr{Z}' $ is regular and proper over $S'$, and fits into a commutative diagram

\begin{equation}\label{ABCD}
\xymatrix{
    \mathscr{Z}'  \ar[r]^{\pi'} \ar[d]_{\phi} & \mathscr{Y}'  \ar[d]^{\varphi} \\
   \mathscr{Z}   \ar[r]^{\pi} & \mathscr{Y}.
}    
\end{equation}

The proper birational morphism
$$ \phi \colon \mathscr{Z}' \to \mathscr{Z} $$
can be factored into a composition of blow-ups at closed points in the special fiber and is the identity away from the finitely many (closed) points $ P \in \mathscr{Z}_k $ with non-trivial stabilizer $G_P \subset G$. In particular, $\mathscr{Z}'$ is an sncd model.

\subsubsection{} \label{subsec:shortexactsequences}
Let $ g \colon \mathscr{Z}' \to S' $ denote the structural map. By (the proof of) \cite[Cor.~1.2]{Bloch} and \cite[Lemma~7.2]{Bloch}, there are short exact sequences 
\begin{equation}\label{A} 0 \to g^* \Omega^1_{S'/S} \to \Omega^1_{\mathscr{Z}'/S} \to \Omega^1_{\mathscr{Z}'/S'} \to 0.
\end{equation}
 
\begin{equation}\label{B}
0 \to (\pi')^* \Omega^1_{\mathscr{Y}'/S} \to \Omega^1_{\mathscr{Z}'/S} \to \Omega^1_{\mathscr{Z}'/\mathscr{Y}'} \to 0. 
\end{equation}
As is part of these references, one has that, in the derived category, $L{\pi'}^{\ast} \Omega^1_{\mathscr{Y}'/S} = {\pi'}^{\ast} \Omega^1_{\mathscr{Y}'/S}.$

\subsubsection{The local situation}
Assume that $ P \in \mathscr{Z}_k $ has non-trivial stabilizer $G_P$ and denote by $Q$ its image in $ \mathscr{Z}/G_P $. Recall from \ref{subsec:weakwildquot} that there is an associated \emph{local} situation
$$ \pi_P \colon \mathscr{Z}_P \to \mathscr{Z}_P/G_P = \mathscr{Y}_Q, $$
with $\mathscr{Z}_P$ a smooth $S'$-model of $ \mathbb{P}^1_{K'} $ and $\mathscr{Y}_Q$ a scheme over $ S_P $.

We again have minimal resolutions with strict normal crossings $ \mathscr{Y}_Q' \to \mathscr{Y}_Q$ and $\mathscr{Z}_P' \to NB(\mathscr{Y}_Q')$, respectively. The two exact sequences above are also valid in this situation.

\subsubsection{Formal completion}

Replacing $\mathscr{Y}$ by $\mathscr{Z}/G_P$ if necessary, we can and will assume that $G = G_P$ to simplify notation. We write $ \widehat{\mathscr{Y}}_Q = \mathrm{Spec}(\widehat{\mathcal{O}}_{\mathscr{Y}, Q})$ and $ \widehat{\mathscr{Z}}_P = \mathrm{Spec}(\widehat{\mathcal{O}}_{\mathscr{Z}, P})$. 

Recall that the minimal desingularization $ \widehat{\mathscr{Y}}_Q' \to \widehat{\mathscr{Y}}_Q $ is the base change of $\mathscr{Y}' \to \mathscr{Y}$. We explain next that a similar statement holds for the minimal desingularization 
$$ \widehat{\mathscr{Z}}_P' \to NB(\widehat{\mathscr{Y}}_Q').  $$

\begin{lemma}\label{lem:complocglob}
The following identites hold:   
\begin{enumerate}

\item $\widehat{\mathscr{Z}}_P = \mathscr{Z} \times_{\mathscr{Y}} \widehat{\mathscr{Y}}_Q $, and

\item $\widehat{\mathscr{Z}}'_P = \mathscr{Z}' \times_{\mathscr{Y}} \widehat{\mathscr{Y}}_Q $.

\end{enumerate}
The analogous statements hold also when replacing $\mathscr{Y}$ and $\mathscr{Z}$ with $\mathscr{Y}_Q$ and $\mathscr{Z}_P$, respectively, in the local situation.
\end{lemma}
\begin{proof}
For both statements, by standard arguments, we can and will replace $\mathscr{Y}$ by $\mathrm{Spec}(\mathcal{O}_{\mathscr{Y}, Q})$.

In order to prove (1), recall that $ \mathscr{Z} = NB(\mathscr{Y})$. Since $ \mathscr{Z} \to \mathscr{Y} $ is finite, and $P$ is the unique point in the preimage of $Q$, the statement follows from the identity
$$ N(\widehat{\mathcal{O}}_{\mathscr{Y}, Q} \otimes_R R') = \widehat{\mathcal{O}}_{\mathscr{Z}, P} $$
(\emph{cf.} \emph{e.g.} \cite[2.2]{Halle-stable}).

We next claim that $NB(\widehat{\mathscr{Y}}_Q') = NB(\mathscr{Y}') \times_{\mathscr{Y}} \mathscr{Y}_Q $. Indeed, $NB(\mathscr{Y}') \to \mathscr{Y}$ is of finite type, and coincides with $\mathscr{Z} \setminus P $ over $\mathscr{Y} \setminus Q$ by what we have already proved. This implies that $NB(\mathscr{Y}') \times_{\mathscr{Y}} \mathscr{Y}_Q$ is regular (and hence normal) on the complement of the fiber over $Q$. It is therefore normal by \cite[54.11.6]{stacks-project}, and necessarily coincides with $NB(\widehat{\mathscr{Y}}_Q')$.

By \cite[54.11]{stacks-project}, the projection $ \widehat{\mathscr{Y}}_Q' \to \mathscr{Y}' $ induces an isomorphism on the level of completed local rings at any point $x$ mapping to $Q$. This fact, together with the finiteness of the respective normalization maps, implies that $ NB(\widehat{\mathscr{Y}}_Q') \to NB(\mathscr{Y}') $ also induces an isomorhism at the level of completed local rings at any point mapping to $Q$ (again arguing as in \cite[2.2]{Halle-stable}). Therefore the pullback of $\mathscr{Z}'$ yields $\widehat{\mathscr{Z}}'_P$ proving (2).

\end{proof}

From Lemma \ref{lem:complocglob} we immediately derive that $\widehat{\mathscr{Z}}'_P = \mathscr{Z}' \times_{\mathscr{Z}} \widehat{\mathscr{Z}}_P$ and that $\widehat{\mathscr{Z}}'_P = \mathscr{Z}' \times_{\mathscr{Y}'} \widehat{\mathscr{Y}}'_Q$. In particular, the pullback of $ \mathscr{Z}' \to \mathscr{Y}' $ to $\widehat{\mathscr{Y}}'_Q$ only depends on $\widehat{\mathscr{Y}}_Q$. The analogous statement is true for $ \mathscr{Z}'_P \to \mathscr{Y}'_Q $.

\subsubsection{Comparing Euler characteristics}\label{subsubsec:Eulercomp}

By the computation in the proof of \cite[Proposition 4.2]{Eri16}, one finds, if $\varphi: \cX' \to \cX$ is a desingularization of a normal model $\cX $,
\begin{equation}\label{eq:Eulercharblowup}
\chi(\cX_k') = \chi(\cX_k) + \sum (\chi(E_i) - 1).
\end{equation}
Here the sum is over the exceptional divisors $E_i$. We find that 
$$ \chi(\mathscr{Y}_k') = \overline{\chi} + \sum_Q a_Q, $$
where $a_Q$ is an integer which only depends on $ \widehat{\mathcal{O}}_{\mathscr{Y}, Q} $. A similar argument for $ \mathscr{Y}'_Q \to \mathscr{Y}_Q $ therefore yields
$$ \chi(\mathscr{Y}_{Q,k}') = 2 + a_Q. $$

Let $b_P$ denote the number of exceptional curves in $ \phi \colon \mathscr{Z}' \to \mathscr{Z} $ mapping to $P$. Then if we write 
$$ \chi(\mathscr{Z}_k') = \chi + \sum_P b_P$$
it follows from Lemma \ref{lem:complocglob} that
$$ \chi(\mathscr{Z}_{P,k}') = 2 + b_P $$
in the local case.

\subsection{Computations with localized Chern classes}

We shall write the localized Chern classes without the capping with the fundamental class $c_i^Z(\mathcal{F}) = c_i^Z(\mathcal{F}) \cap [\cX]$ for the degree of the  $i$-th localized Chern class of a coherent sheaf of finite homological dimension on 
 a scheme $\cX \to S$, locally free away from the $Z \subseteq \cX_k$, as in \ref{subsec:localizedChern}. Since will always work on regular schemes, all coherent sheaves have finite homological dimensions. For top degree classes we also confuse them implicitly with their degrees. 

Below we shall frequently use \cite[Theorem]{Bloch} which states that if $\cX$ is regular, flat, proper and of relative dimension $d$ over $S$, when $d \leq 1,$
\begin{equation}\label{eq:Blochformula}  c^{\cX_k}_{d+1}(\Omega^1_{\cX/S}) = (-1)^d \mathrm{Art}_{\cX/S}.
\end{equation}

\subsubsection{}
We record the following identities, where the notation is as in the previous sections. 
\begin{lemma}\label{lemma:cherncomp1}
The following identities hold:
\begin{enumerate}
\item $ c_2^{\mathscr{Z}'_k} (\Omega^1_{\mathscr{Z}'/S}) = - (e - 1 + \mathbf{sw}) \cdot \chi + \sum_P b_P $, and

\item $ c_2^{\mathscr{Z}'_{P,k}}(\Omega^1_{\mathscr{Z}'_P/S_P}) = - (e_P - 1 + \mathbf{sw}_Q) \cdot 2 + b_P $.
\end{enumerate}
\end{lemma}
\begin{proof}
We claim that Bloch's theorem recalled above yields $ c_1^{s}(\Omega^1_{S'/S}) = (e - 1 + \mathbf{sw}) $ and $ c_2^{\mathscr{Z}'_k} (\Omega^1_{\mathscr{Z}'/S'}) = \sum_P b_P $, respectively. Indeed, in the former case it is the conductor-discriminant formula, and in the latter case, we use that $\mathscr{Z}/S'$ is smooth and that 
$$ \chi(\mathscr{Z}'_k) = \chi(\mathscr{Z}_k) + \sum_P b_P. $$
Identity (1) therefore follows immediately from applying \cite[Prop.~1.1 (iii)]{Bloch} to Sequence \eqref{A}. The proof of (2) is similar.

\end{proof}

\begin{lemma}\label{lemma:cherncomp2}
The following identities hold:
\begin{enumerate}
\item $ c_2^{\mathscr{Z}'_k}((\pi')^*\Omega^1_{\mathscr{Y}'/S}) = e \cdot (\sum_Q a_Q + \overline{\chi} - \chi + \mathrm{Sw}(C))$, and

\item $ c_2^{\mathscr{Z}'_{P,k}}( (\pi'_P)^* \Omega^1_{ \mathscr{Y}'_Q/S_P}) = e_P \cdot a_Q $.
\end{enumerate}
\end{lemma}
\begin{proof}
We only prove the first assertion, the second follows the same line of reasoning. It is deduced from the identity
$$ \mathrm{deg}~c_2^{\mathscr{Z}'_k}((\pi')^*\Omega^1_{\mathscr{Y}'/S}) \cap [\mathscr{Z}'] = \mathrm{deg}~c_2^{\mathscr{Y}'_k}(\Omega^1_{\mathscr{Y}'/S}) \cap e \cdot [\mathscr{Y}'], $$
where the latter term equals $ e \cdot \mathrm{Art}_{\mathscr{Y}'/S}$ by \eqref{eq:Blochformula}. This follows from the projection formula for $L{\pi'}^*\Omega^1_{\mathscr{Y}'/S} = (\pi')^*\Omega^1_{\mathscr{Y}'/S}$ (cf. \textsection \ref{subsec:shortexactsequences}) and the fact that $\pi'$ is generically finite of degree $e$. 

\end{proof}

\subsubsection{}
By the Whitney formula in \cite[Prop.~1.1]{Bloch}, Sequence \eqref{B} gives the identity
\begin{equation}\label{eq:whitneyc2}
c_2^{\mathscr{Z}'_k}(\Omega^1_{\mathscr{Z}'/S}) = c_2^{\mathscr{Z}'_k}((\pi')^*\Omega^1_{\mathscr{Y}'/S}) + c_2^{\mathscr{Z}'_k}(\Omega^1_{\mathscr{Z}'/\mathscr{Y}'}) + c_1((\pi')^*\Omega^1_{\mathscr{Y}'/S}) \cdot c_1^{\mathscr{Z}'_k}(\Omega^1_{\mathscr{Z}'/\mathscr{Y}'}). 
\end{equation}

The last two terms in the above expression are difficult to control in general. We will need the following crucial computation of these in terms of local contributions over the standard weak wild models:

\begin{prop}\label{Prop:Lloc}
Let $P$ run over the points in $\mathscr{Z}$ with non-trivial stabilizer. Then the sum 
$$ c_2^{\mathscr{Z'}_k}(\Omega^1_{\mathscr{Z}'/S}) - c_2^{\mathscr{Z'}_k}((\pi')^*\Omega^1_{\mathscr{Y}'/S}) $$
equals the corresponding sum over the  weak wild models as follows:
$$\sum_P \left( c_2^{\mathscr{Z}'_{P,k}}(\Omega^1_{\mathscr{Z}'_P/S}) - c_2^{\mathscr{Z}'_{P,k}}((\pi'_{P})^*\Omega^1_{\mathscr{Y}'_Q/S_P})\right). $$

\end{prop}
\begin{proof}
Let us denote by $ E_P $ the exceptional locus over $P$ in $ \phi \colon \mathscr{Z}' \to \mathscr{Z}$. As explained in Lemma \ref{lem:complocglob}, this is also the exceptional locus of $ \widehat{\mathscr{Z}}'_P \to \widehat{\mathscr{Z}}_P$, which is moreover obtained by base change of $\phi $ to the completion at $P$. The sheaf $\Omega_{\cZ'/\cY'}$ is trivial on the complement of $ \sqcup_P E_P $.

For each $P$, there exists a suitable open $ U_P \subset \mathscr{Z}'$ containing $E_P$ and having empty intersection with any other $E_{P'}$. Then by the first two properties recalled in Proposition \ref{prop:chernclassproperties}
$$ c_i^{\mathscr{Z}'_{k}}(\Omega_{\cZ'/\cY'}) = \sum_P c_i^{E_P}(\Omega_{\cZ'/\cY'}) \cap [U_P], $$
and we claim that
\begin{equation}\label{eq:betai} \beta_{i,P} := c_i^{E_P}(\Omega_{\cZ'/\cY'}) \cap [U_P] = c_i^{E_P}(\Omega_{\cZ_P'/\cY_Q'}). 
\end{equation}

Because the diagram
\begin{equation}\label{ABCD}
\xymatrix{
    \widehat{\mathscr{Z}}'_P  \ar[r]\ar[d] & \widehat{\mathscr{Y}}'_Q  \ar[d] \\
   \mathscr{Z'}   \ar[r] & \mathscr{Y'}
}
\end{equation}
is Cartesian by Lemma \ref{lem:complocglob}, we have
\begin{equation}\label{eq:pullbackjpOmega} j_P^*\Omega^1_{\mathscr{Z}'/\mathscr{Y}'} = \Omega^1_{\widehat{\mathscr{Z}}'_P/\widehat{\mathscr{Y}}'_Q},  
\end{equation}
where $j_P: \mathscr{Z}'_P \to  \mathscr{Z}'$ is the natural map. An argument analogous to that of the proof of Proposition \ref{prop:chernclassproperties} \eqref{item4:chernclassproperties}  shows that the $\beta_{i,P}$ only depend on the completions. Since the same argument applies to $\Omega_{\mathscr{Z}'_P/\mathscr{Y}'_Q}$ which also base changes to $\Omega^1_{\widehat{\mathscr{Z}}'_P/\widehat{\mathscr{Y}}'_Q}$ the formula in \eqref{eq:betai} is  proven. 

For the rest of the proof, note that the sum we want to relate to the weak wild models, by \eqref{eq:whitneyc2}, consists of the sum  over $P$ of the terms
$$\beta_{2,P} + c_1((\pi')^*\Omega^1_{\mathscr{Y}'/S} ) \cap \beta_{1,P}. $$

 We have already proven part of the statement for the  $\beta_{i,P}$. To finish, we must show that the term $c_1((\pi')^*\Omega^1_{\mathscr{Y}'/S} ) \cap \beta_{1,P}$ only depends on the completion, in a reasonable way. For this, we first note that by Sequences \eqref{A} and \eqref{B}, as well the exact sequence 
$$ 0 \to \mathscr{\phi}^*\Omega^1_{\mathscr{Z}/S'} \to \Omega^1_{\mathscr{Z}'/S'} \to \Omega^1_{\mathscr{Z}'/\mathscr{Z}} \to 0, $$
so that we have 
\begin{equation}\label{eq:c1computationindependence} c_1((\pi')^*\Omega^1_{\mathscr{Y}'/S}) = c_1(\Omega^1_{\mathscr{Z}'/\mathscr{Y}'}) - c_1(\phi^* \Omega^1_{\mathscr{Z}/S'}) + c_1(\Omega^1_{\mathscr{Z}'/\mathscr{Z}}). 
\end{equation}
Indeed, this follows from additivity of the first Chern class, together with the fact that $ c_1(g^*\Omega^1_{S'/S}) = 0 $, since $\Omega^1_{S'/S}$ has a finite free resolution by modules of rank one on $S'$. 
We show that the right hand of \eqref{eq:c1computationindependence}, when restricted to a neighborhood of $E_P$, only depends on the formal completion in a controlled fashion. 

The inclusion $k: E_P \to \mathscr{Z}' $ factors via $\widehat{k}: E_P \to \widehat{\mathscr{Z}}'_P$ and $ j_P \colon \widehat{\mathscr{Z}}'_P \to \mathscr{Z}' $. Since $j_P$ is flat,  $Lk^{\ast} = L\widehat{k}^{\ast } Lj_P^{\ast} = L\widehat{k}^{\ast } j_P^{\ast} $. Also, for a coherent sheaf $\mathcal{F}$ on $\mathscr{Z}'$, one has $$c_1(\mathcal{F}) \cap \beta_{1,P} = c_1(Lk^{\ast} \mathcal{F}) \cap \beta_{1,P} = c_1(L\widehat{k}^{\ast} j_P^\ast \mathcal{F}) \cap \beta_{1,P} = c_1(j_P^\ast \mathcal{F}) \cap \beta_{1,P}.$$
This reduces the statement to prove that the application of $j_P^{\ast}$ to each of the three sheaves appearing in the right hand side of \eqref{eq:c1computationindependence} only depend on $\widehat{\mathscr{Z}}'_P$. For $\Omega_{\mathscr{Z}'/\mathscr{Y}'}$, this was already noted in the discussion surrounding \eqref{eq:pullbackjpOmega}. A similar argument shows that 
$$ j_P^*\Omega^1_{\mathscr{Z}'/\mathscr{Z}} = \Omega^1_{\widehat{\mathscr{Z}}'_P/\widehat{\mathscr{Z}}_P}.  $$ For $\phi^* \Omega^1_{\mathscr{Z}/S'}$, this is because the pullback of $ \Omega^1_{\mathscr{Z}/S'}$ to $\widehat{\mathscr{Z}}_P$ only depends on the completion, as in the proof of Propositon \ref{prop:chernclassproperties}).

Repeating the argument for $(\pi'_P)^*\Omega^1_{\mathscr{Y}'_Q/S_P}$ therefore yields exactly the same description.

\end{proof}

\subsection{Proof using localized Chern classes}
Now we give the proof of Proposition \ref{Prop:KeyIden}.

\begin{proof}
Using Proposition \ref{Prop:Lloc} and Lemma \ref{lemma:cherncomp2} we can write
$$ c_2^{\mathscr{Z}'_k}(\Omega^1_{\mathscr{Z}'/S}) - e \cdot (\sum_Q a_Q + \overline{\chi} - \chi + \mathrm{Sw}(C)) $$
$$ = \sum_P \left(c_2^{\mathscr{Z}'_{P,k}}(\Omega^1_{\mathscr{Z}'_P/S_P}) - e_P \cdot a_Q  \right)$$

Combining with Lemma \ref{lemma:cherncomp1}, this becomes
$$  \sum_P b_P - (e - 1 + \mathbf{sw}) \cdot \chi - e \cdot (\sum_Q a_Q + \overline{\chi} - \chi + \mathrm{Sw}(C)) =  $$
$$ \sum_P \left( b_P - (e_P - 1 + \mathbf{sw}_Q) \cdot 2 - e_P \cdot a_Q  \right).  $$
For every $P$ mapping to a singular point $Q$ there are $e/e_P$ points in the $G$-orbit of $P$, hence all terms involving the numbers $a_Q$ cancel. Therefore the identity becomes (after also dividing by $e$ on both sides)
$$ - (e - 1 + \mathbf{sw}) \cdot \frac{\chi}{e} - (\overline{\chi} - \chi + \mathrm{Sw}(C)) = $$
$$ - 2 \sum_P \left( \frac{e_P}{e} - \frac{1}{e} + \frac{\mathbf{sw}_Q}{e} \right). $$

In the sequel we use the notation for points of type $i$ introduced in Definition \ref{def:pointtypei}. The right hand side can be written
$$ - 2 \sum_{i=0}^{r-1} \sum_{j=1}^{m_i} p^i \left( \frac{p^{r-i}}{p^r} - \frac{1}{p^r} + \frac{\mathbf{sw}_Q}{p^r} \right) = $$
$$ - 2 \sum_{i=0}^{r-1} m_i \left( 1 - \frac{1}{p^{r-i}} \right) - 2 \sum_{i=0}^{r-1} \sum_{j=1}^{m_i} \frac{\mathbf{sw}_{Q_{i,j}}}{p^{r-i}}. $$

Recall that the Riemann-Hurwitz formula for $ \mathscr{Z}_k \to \mathscr{Z}_k/G $ states that 
$$ - \frac{\chi}{e} = - \overline{\chi} + 2 \sum_{i=0}^{r-1} m_i \left( 1 - \frac{1}{p^{r-i}}  \right). $$
Using this relation, we compute the left hand side as follows:
$$ \left( \frac{\chi}{e} - \overline{\chi} \right) + \mathbf{sw} \cdot \left( - \overline{\chi}  + 2 \sum_{i=0}^{r-1} m_i \left( 1 - \frac{1}{p^{r-i}}\right) \right) - \mathrm{Sw}(C) = $$
$$ - 2 \sum_{i=0}^{r-1} m_i \left( 1 - \frac{1}{p^{r-i}}  \right) + \mathbf{sw} \cdot \left( - \overline{\chi}  + 2 \sum_{i=0}^{r-1} m_i \left( 1 - \frac{1}{p^{r-i}}\right) \right) - \mathrm{Sw}(C). $$

We can now immediately derive the identity
$$ - \mathrm{Sw}(C) = - 2 \sum_{i=0}^{r-1} \sum_{j=1}^{m_i} \frac{\mathbf{sw}_{Q_{i,j}}}{p^{r-i}} + \mathbf{sw} \cdot \left( \overline{\chi}  - 2 \sum_{i=0}^{r-1} m_i \left( 1 - \frac{1}{p^{r-i}}\right) \right), $$
which is what we needed to prove.

\end{proof}

\subsection{Direct argument}

We briefly explain a different approach to Proposition \ref{Prop:KeyIden}. The key ingredients are: 

(1) The Swan conductor $\mathrm{Sw}(C)$ can formally be expressed in terms of data associated with the filtration by the higher ramification subgroups $G_i \subset G$, and the abelian rank $a_i$ of $ C \times_K (K')^{G_i}$, for each $i$. 

(2) The same data also appears in the Riemann-Hurwitz formula for each of the two maps in the factorization
$$ \mathscr{Z}_k \to \mathscr{Z}_k/G_i \to \mathscr{Z}_k/G $$
(one checks that the action of $G/G_i$ is again weakly ramified), in addition to ramification data for the action of $G$ and $G_i$, respectively. 

(3) The latter is linked to the "local" Swan conductors $ \mathbf{sw}_Q$, which can be formally expressed in terms of the ramification filtration of $K'/(K')^{G_P}$, where $P$ is a preimage of $Q$.

With this input, it is possible to verify the identity in Proposition \ref{Prop:KeyIden} through a lengthy direct computation.

\section*{Acknowledgements}

The authors are grateful for interesting conversations related to this project with many colleagues, especially Christian Johansson and Gerard Freixas i Montplet. We moreover thank Tong Zhang for, in particular, bringing \cite{SLTan} to our attention. 

The first author is supported by the Swedish Research Council, VR grant 2021-03838 "Mirror symmetry in genus one". 

\bibliographystyle{amsplain}
\bibliography{main}{}

\providecommand{\bysame}{\leavevmode\hbox to3em{\hrulefill}\thinspace}
\providecommand{\MR}{\relax\ifhmode\unskip\space\fi MR }
\providecommand{\MRhref}[2]{%
  \href{http://www.ams.org/mathscinet-getitem?mr=#1}{#2}
}
\providecommand{\href}[2]{#2}
\begin{thebibliography}{10}

\bibitem{SGA71}
\emph{Groupes de monodromie en g\'{e}om\'{e}trie alg\'{e}brique. {I}}, Lecture Notes in Mathematics, vol. Vol. 288, S\'{e}minaire de G\'{e}om\'{e}trie Alg\'{e}brique du Bois-Marie 1967--1969 (SGA 7 I), Dirig\'{e} par A. Grothendieck. Avec la collaboration de M. Raynaud et D. S. Rim.

\bibitem{BakNic}
M.~Baker and J.~Nicaise, \emph{Weight functions on {B}erkovich curves}, Algebra Number Theory \textbf{10} (2016), no.~10, 2053--2079. \MR{3582013}

\bibitem{Bloch}
S.~Bloch, \emph{Cycles on arithmetic schemes and euler characteristics of curves}, Algebraic geometry, Bowdoin, Proc. Symp. Pure Math., vol.~46, Amer. Math. Soc., Providence, RI, 1985, pp.~421--450.

\bibitem{neron}
S.~Bosch, W.~L\"utkebohmert, and M.~Raynaud, \emph{N\'eron models}, Ergebnisse der Mathematik und ihrer Grenzgebiete, vol.~21, Springer-Verlag, 1990.

\bibitem{chai}
C.-L. Chai, \emph{N\'eron models for semiabelian varieties: congruence and change of base field}, Asian J. Math. \textbf{4} (2000), no.~4, 715--736.

\bibitem{ChaiKappen}
C.-L. Chai and C.~Kappen, \emph{A refinement of the {A}rtin conductor and the base change conductor}, Algebr. Geom. \textbf{2} (2015), no.~4, 446--475.

\bibitem{chai-yu}
C.-L. Chai and J.-K. Yu, \emph{Congruences of n\'eron models for tori and the artin conductor}, Ann. Math. (2) \textbf{154} (2001), 347--382.

\bibitem{DelSGA7}
P.~Deligne, \emph{Intersections sur les surfaces r{\'e}guli{\`e}res, in}, Springer-Verlag, Berlin-New York, S\'eminaire de G\'eom\'etrie Alg\'ebrique du Bois-Marie 1967--1969 (SGA 7 II), Dirig\'e par P. Deligne et N. Katz.

\bibitem{detcoh}
\bysame, \emph{Le d\'{e}terminant de la cohomologie}, Current trends in arithmetical algebraic geometry ({A}rcata, {C}alif., 1985), Contemp. Math., vol.~67, Amer. Math. Soc., Providence, RI, 1987, pp.~93--177.

\bibitem{EdixhovenTame}
B.~Edixhoven, \emph{N\'eron models and tame ramification}, Compositio Math. \textbf{81} (1992), no.~3, 291--306.

\bibitem{Eri16}
D.~Eriksson, \emph{Discriminants and {A}rtin conductors}, J. Reine Angew. Math. \textbf{712} (2016), 107--121. \MR{3466549}

\bibitem{ErikssonErratum}
\bysame, \emph{Erratum to {D}iscriminants and {A}rtin conductors ({J}. reine angew. {M}ath. 712 (2016), 107--121)}, J. Reine Angew. Math. \textbf{814} (2024), 283--284.

\bibitem{spectralgenus}
D.~Eriksson and G.~Freixas~i Montplet, \emph{The spectral genus of an isolated hypersurface singularity and a conjecture relating to the milnor number},  (2024).

\bibitem{CDG-1}
D.~Eriksson, G.~Freixas~i Montplet, and C.~Mourougane, \emph{Singularities of metrics on {H}odge bundles and their topological invariants}, Algebr. Geom. \textbf{5} (2018), no.~6, 742--775.

\bibitem{CDG-2}
\bysame, \emph{B{COV} invariants of {C}alabi-{Y}au manifolds and degenerations of {H}odge structures}, Duke Math. J. \textbf{170} (2021), no.~3, 379--454.

\bibitem{logjumps}
D.~Eriksson, L.H. Halle, and J.~Nicaise, \emph{A logarithmic interpretation of {E}dixhoven's jumps for {J}acobians}, Adv. Math. \textbf{279} (2015), 532--574.

\bibitem{Faltings}
G.~Faltings, \emph{Endlichkeitss\"{a}tze f\"{u}r abelsche {V}ariet\"{a}ten \"{u}ber {Z}ahlk\"{o}rpern}, Invent. Math. \textbf{73} (1983), no.~3, 349--366.

\bibitem{Halle-neron}
L.H. Halle, \emph{Galois actions on n\'eron models of jacobians}, Ann. Inst. Fourier \textbf{60} (2010), no.~3, 853--903.

\bibitem{Halle-stable}
\bysame, \emph{Stable reduction of curves and tame ramification}, Math. Zeit. \textbf{265} (2010), no.~3, 529--550.

\bibitem{JumpsMonodromy}
L.H. Halle and J.~Nicaise, \emph{Jumps and monodromy of abelian varieties}, Doc. Math. \textbf{16} (2011), 937--968.

\bibitem{HaNi}
\bysame, \emph{Motivic zeta functions of abelian varieties, and the monodromy conjecture}, Adv. Math. \textbf{227} (2011), 610--653.

\bibitem{HaNi-book}
\bysame, \emph{N\'eron models and base change}, Lecture Notes in Mathematics, vol. 2156, Springer, 2016.

\bibitem{ItoSchroer}
H.~Ito and S.~Schr{\"o}er, \emph{Wild quotient surface singularities whose dual graphs are not star-shaped}, Asian J.~Math. \textbf{19} (2015), no.~5, 951--986.

\bibitem{KatoSaito}
K.~Kato and T.~Saito, \emph{On the conductor formula of {B}loch}, Publ. Math. Inst. Hautes \'{E}tudes Sci. (2004), no.~100, 5--151.

\bibitem{KiVi}
K.~Kiyek and J.~L. Vicente, \emph{Resolution of curve and surface singularities. in characteristic zero}, Algebras and Applications, vol.~4, Kluwer Academic Publishers, Dordrecht, 2004.

\bibitem{KMdeterminant}
F.~Knudsen and D.~Mumford, \emph{The projectivity of the moduli space of stable curves. {I}. {P}reliminaries on ``det'' and ``{D}iv''}, Math. Scand. \textbf{39} (1976), no.~1, 19--55.

\bibitem{Kollar}
J.~Koll\'ar, \emph{Singularities of pairs}, Algebraic geometry---{S}anta {C}ruz 1995, Proc. Sympos. Pure Math., vol.~62, Amer. Math. Soc., Providence, RI, 1997, pp.~221--287. \MR{1492525}

\bibitem{Laufer}
H.~Laufer, \emph{On {$\mu $} for surface singularities}, Several complex variables ({P}roc. {S}ympos. {P}ure {M}ath., {V}ol. {XXX}, {P}art 1, {W}illiams {C}oll., {W}illiamstown, {M}ass., 1975), Proc. Sympos. Pure Math., vol. Vol. XXX, Part 1, Amer. Math. Soc., Providence, RI, 1977, pp.~45--49.

\bibitem{Lipman}
J.~Lipman, \emph{Rational singularities, with applications to algebraic surfaces and unique factorization}, Inst. Hautes \'{E}tudes Sci. Publ. Math. (1969), no.~36, 195--279.

\bibitem{liu}
Q.~Liu, \emph{Algebraic geometry and arithmetic curves}, Oxford Graduate Texts in Mathematics, vol.~6, Oxford University Press, 2002.

\bibitem{Lor90}
D.~Lorenzini, \emph{Dual graphs of degenerating curves}, Math. Ann. \textbf{287} (1990), no.~1, 135--150.

\bibitem{lorenzini-wildsurfacequot}
\bysame, \emph{Wild quotient singularities of surfaces}, Math. Z. \textbf{275} (2013), 211--232.

\bibitem{Lor14}
\bysame, \emph{Wild models of curves}, Algebra Number Theory \textbf{8} (2014), no.~2, 331--367.

\bibitem{MitSme}
K.~Mitsui and A.~Smeets, \emph{Logarithmic good reduction and the index}, 1711.11547.

\bibitem{JNgeometriccriteria}
J.~Nicaise, \emph{Geometric criteria for tame ramification}, Math. Z. \textbf{273} (2013), no.~3-4, 839--868.

\bibitem{ObWe}
A.~Obus and S.~Wewers, \emph{Explicit resolution of weak wild quotient singularities of arithmetic surfaces}, J. Algebraic Geom. \textbf{29} (2020), no.~4, 691--728.

\bibitem{saito-VC}
T.~Saito, \emph{Vanishing cycles and geometry of curves over a discrete valuation ring}, Am. J. Math. \textbf{109} (1987), no.~6, 1043--1085.

\bibitem{Saitodisc}
\bysame, \emph{Conductor, discriminant, and the {N}oether formula of arithmetic surfaces}, Duke Math. J. \textbf{57} (1988), no.~1, 151--173.

\bibitem{stacks-project}
The {Stacks Project Authors}, \emph{\textit{Stacks Project}}, \url{https://stacks.math.columbia.edu}, 2023.

\bibitem{SLTan}
S.-L. Tan, \emph{On the invariants of base changes of pencils of curves. {I}}, Manuscripta Math. \textbf{84} (1994), no.~3-4, 225--244.

\bibitem{SLTan2}
\bysame, \emph{On the invariants of base changes of pencils of curves. {II}}, Math. Z. \textbf{222} (1996), no.~4, 655--676.

\bibitem{weatershoot2024thesis}
A.~Waeterschoot, \emph{Wild ramification of $p$-adic analytic curves}, upcoming phd thesis, KU Leuven, Leuven, Belgium, 2025.

\bibitem{winters}
G.~B. Winters, \emph{On the existence of certain families of curves}, Am. J. Math. \textbf{96} (1974), 215--228.

\end{thebibliography}

\end{document}